\documentclass[11 pt]{article}

\usepackage[utf8]{inputenc}
\usepackage[english]{babel}   
\usepackage{graphicx}
\usepackage{lmodern}
\usepackage{amsmath}
\usepackage{tikz}
\usetikzlibrary{calc}
\usepackage{longtable}
\usepackage{epsfig}
\usepackage{color}
\usepackage{setspace}
\usepackage[top=2cm,bottom=2.5cm,left=2.5cm,right=2.5cm]{geometry}
\usepackage{sectsty}
\usepackage{fancyhdr}
\usepackage{wasysym}
\usepackage[colorlinks=true]{hyperref}
\usepackage{caption}
\usepackage{subcaption}
\usetikzlibrary{patterns}
\usepackage{amsfonts}
\usepackage{amssymb}
\usepackage{amsthm}
\usepackage{enumerate}
\usepackage{bm}
\usepackage{todonotes}
\usepackage{algorithm}
\usepackage{algpseudocode}
\usepackage{comment}

\newtheorem{theorem}{Theorem} %[subsection]
\newtheorem{definition}[theorem]{Definition}
\newtheorem{lemma}[theorem]{Lemma}
\newtheorem{proposition}[theorem]{Proposition}
\newtheorem{remark}[theorem]{Remark}

\newtheorem{coro}[theorem]{Corollary}

\newcommand\tr{\mathsf T}

\newcommand\Exp{\operatorname{Exp}}
\newcommand\sym{\operatorname{sym}}
\newcommand\Log{\operatorname{Log}}
\newcommand{\dd}{\mathop{}\!\mathrm{d}}
\newcommand\Grad{\operatorname{Grad}}
\newcommand\Ker{\operatorname{Ker}}
\newcommand{\DG}{\mathrm{DG}}

\graphicspath{{./figures}}

\title{Proximal Galerkin for the isometry constraint}
\date{}
\author{\begin{minipage}{\textwidth}\centering
		Brendan Keith
      \\
      {\small Division of Applied Mathematics, Brown University, Providence, RI 02912, USA}\\
     {\small email: \texttt{brendan\_keith@brown.edu}}
      \\[1em]
      {\small and}
      \\[1em]
      Frédéric Marazzato \\
		{\small Department of Mathematical Sciences, University of Nevada Las Vegas, NV 89154-4020, USA}\\
  {\small email: \texttt{frederic.marazzato@unlv.edu}}
   \end{minipage}
   }
      
\begin{document}
\hypersetup{urlcolor=blue,linkcolor=red,citecolor=blue}

\maketitle

\begin{abstract}
    We resolve a longstanding open problem in the computational modeling of nonlinear plates by introducing a numerical method that exactly enforces the isometry constraint, namely, that the first fundamental form of the mid-surface coincides with the identity tensor.
    Several numerical methods have been proposed to approximate solutions of such manifold-constrained variational problems using gradient flows with tangent space updates.
However, this class of methods presents two main challenges.
First, a preprocessing step is required to enforce the boundary conditions and generate an initial guess sufficiently close to an isometry. Second, each step of the gradient flow typically increases the isometry defect.
We adopt an alternative approach based on the proximal Galerkin framework, originally introduced for variational problems with convex inequality constraints.  
The resulting method preserves the geometric structure of the feasible set and yields an efficient algorithm in which each iterate is an exact isometry at the barycenter of every mesh cell.
In contrast to existing methods, no preprocessing step is required, enabling broader applicability of this important category of mathematical models.
Numerical experiments on standard benchmarks demonstrate that the method converges to a prescribed error tolerance in an asymptotically mesh-independent number of iterations and requires substantially fewer iterations than previous methods, even on coarse meshes.
\end{abstract}

\section{Introduction}
The seminal work \cite{friesecke2006hierarchy} shows that nonlinear elasticity gives rise to distinct two-dimensional plate models in different asymptotic regimes.
We consider large bending deformations of Kirchhoff plates under an isometry constraint, a model relevant, for instance, to the deformation of paper \cite{MR4827936,MR4524370}.

To introduce the model, we let $\Omega \subset \mathbb{R}^2$ be a bounded, polygonal domain with Lipschitz boundary $\partial \Omega$.
A deformation variable $\bm{y} : \Omega \to \mathbb{R}^3$ describes the midsurface of a plate of given non-zero thickness.
Given Dirichlet boundary conditions and a body force $\bm{f} : \Omega \to \mathbb{R}^3$, we seek isometric deformations satisfying
\begin{subequations}
\label{eq:EnergyPrinciple}
\begin{equation}
\label{eq:IsometryConstraint}
\nabla \bm{y}^\top \nabla \bm{y} = I \quad \text{a.e.\ in } \Omega,
\end{equation}
that minimize
\begin{equation}
\label{eq:EnergyFunctional}
E(\bm{y}) := \frac{1}{2} \int_\Omega \nabla^2 \bm y  : \nabla^2 \bm y \dd x - \int_\Omega \bm f \cdot \bm y \dd x
\end{equation}
\end{subequations}
Here, $\nabla^2 \bm{y}$ denotes the Hessian, while $:$ denotes the Frobenius product of tensors.
The model is thus an energy principle with a nonlinear PDE constraint.

Numerous discretization strategies have been proposed for this model.
Conforming Kirchhoff elements were introduced in \cite{bartels2013approximation}.
A symmetric interior penalty discontinuous Galerkin (IPDG) method, which is simpler to implement, was proposed in \cite{bonito2021dg}.
More advanced schemes also exist \cite{bonito2023numerical,MR4839135}, but are not considered here, as the main difficulty remains the enforcement of the isometry constraint~\eqref{eq:IsometryConstraint}.

To date, the literature has focused on discrete $H^2$ gradient flows to minimize~\eqref{eq:EnergyFunctional} while restricting updates to the tangent space of the isometry manifold.
These methods require an initial guess that satisfies the boundary conditions and has a small isometry defect.
Constructing such a guess is costly and nontrivial in this setting.
Moreover, since iterates lie in the tangent space rather than on the manifold, they tend to drift away from the feasible set, producing an \textit{isometry defect} that must be considered during analysis.
This defect can be reduced by decreasing the pseudo-time step size, but doing so can severely limits performance.
Perhaps more detrimentally, the convergence of such schemes requires the step size to depend on the mesh size, leading to mesh-dependent iteration complexity, with the number of iterations diverging as the mesh size tends to zero.

In this work, we instead adopt a proximal Galerkin approach \cite{keith2023proximal,keith2025apriori,ern2026proximal}.
This allows us to overcome each of the aforementioned issues, while retaining the relatively simple IPDG discretization.
Our method is inspired by Riemannian proximal point algorithms \cite{ferreira2002proximal} and exploits the Stiefel manifold geometry of the isometry constraint~\eqref{eq:IsometryConstraint}.

Like other proximal Galerkin methods \cite{papadopoulos2024hierarchical,dokken2025latent,fu2026proximal,fu2026locally,castillon2026proximal}, ours amounts to applying a Galerkin method to solve a recursive system of smooth, nonlinear PDEs, the solutions of which are feasible iterates that converge to minimizers of~\eqref{eq:EnergyPrinciple}.
After deriving this system of PDEs and introducing our chosen discretization, we prove existence and uniqueness of the discrete iterates and show that the discrete energy decreases at every new iterate.
The resulting iterates satisfy the isometry constraint at cell barycenters as well as boundary conditions, thus eliminating the need for a preprocessing step.
Moreover, the pseudo-time step is independent of the mesh size, allowing large steps on fine meshes, with our experiments demonstrating \textit{mesh-independent} iteration complexity.

The paper is organized as follows.
Section \ref{sec:preliminaries} recalls elementary Riemannian geometry and properties of the Stiefel manifold.
Section~\ref{sec:continuous} formulates the continuous problem and derives the proximal Galerkin iteration.
Section~\ref{sec:discrete} presents the IPDG discretization, establishes well-posedness and discrete energy decay. % and convergence to stationary points.
Section~\ref{sec:convergence} proves convergence of global discrete minimizers via compactness and $\Gamma$-convergence.
Finally, Section~\ref{sec:numerics}, demonstrates improved performance on standard benchmarks.

\section{Preliminaries}
\label{sec:preliminaries}
This section recalls a few basic facts about Riemannian manifolds.
One can refer to \cite{MR3887684} and \cite{zimmerman2022computing} for further details.

\subsection{Riemannian manifolds: A crash course}
Let $\mathcal M$ be a smooth manifold.
We denote the tangent space at a point $p \in \mathcal M$ by $T_p\mathcal M$, and its dual space, called the cotangent space, by $T_p^\ast \mathcal M$.
The tangent bundle is the disjoint union of all tangent spaces to the manifold, namely $T\mathcal{M} := \{ (p,\mu) \mid p \in \mathcal{M},\, \mu \in  T_p \mathcal{M}\}$.
The cotangent bundle is defined analogously, $T^\ast\mathcal{M} := \{ (p,\xi) \mid p \in \mathcal{M},\, \xi \in  T_p^\ast \mathcal{M}\}$.
Following standard convention, we identify members of the tangent bundle with $T_p\mathcal M$ and use the notations $(p,\mu)$, $\mu_p$, and $\mu$ interchangeably for elements of $T_p\mathcal M$, depending on the desired emphasis on the base point $p$.  

The differential of a smooth function $f \colon \mathcal M \to \mathbb R$ at $p \in \mathcal M$, denoted by $df_p$, is an element of $T_p^\ast \mathcal M$.
In particular, $df_p(\mu) \in \mathbb{R}$ is the directional derivative of $f$ at $p$ in the direction $\mu \in T_p\mathcal M$.
In this notation, the differential of $f$ may be viewed as a section of the cotangent bundle, that is, as a map
\[
    df \colon \mathcal M \to T^\ast \mathcal M
    ,
    \qquad
    p \mapsto df_p
    .
\]

Letting $(\mathcal M,g)$ be a smooth Riemannian manifold, the Riemannian metric at $p \in \mathcal M$ is an inner-product $g_p(\cdot,\cdot) \colon T_p \mathcal{M} \times T_p \mathcal{M} \to \mathbb{R}$ providing a natural isomorphism $\mu \mapsto g_p(\mu,\cdot) \in T_p^\ast \mathcal{M}$.
An important use of this isomorphism is to define the {gradient} of a smooth function on $\mathcal M$.
\begin{definition}[Gradient]
\label{def:gradient}
Let $(\mathcal M, g)$ be a Riemannian manifold and $f \colon \mathcal M \to \mathbb{R}$ be a smooth function.
The gradient of $f$ is the unique vector field $\operatorname{Grad} f \in T\mathcal M$ satisfying
\begin{equation}
\label{eq:GradDefinition}
    g(\operatorname{Grad} f,\cdot) = df
    .
\end{equation}
\end{definition}

\noindent
The metric can also be used to define the magnitude of a vector $\mu \in T_p \mathcal M$,
\( |\mu|_p := \sqrt{g_p(\mu, \mu)} \),
which appears in the definition of the length of a curve in $\mathcal M$.
\begin{definition}[Length of a curve]
\label{def:length curve}
Let $(\mathcal M, g)$ be a Riemannian manifold.
Let $a < b$ and $\gamma : [a,b] \to \mathcal M$ be a piecewise-smooth curve.
The length of $\gamma$ is defined to be
\[ L(\gamma) := \int_a^b |\dot{\gamma}|_{\gamma(t)} \dd t. \]
\end{definition}

\noindent
Having defined the length of a curve, we can now define the distance between points on $\mathcal M$.
\begin{definition}[Riemannian distance]
\label{th:distance}
Let $(\mathcal M, g)$ be a Riemannian manifold.
The Riemannian distance between $p,q \in \mathcal M$ is defined to be
\[ \operatorname{Dist}(p, q) := \inf_{\gamma \in S_{p,q}} L(\gamma), \]
where $S_{p,q}$ is the set of piecewise-smooth curves $\gamma:[0,1] \to \mathcal M$ satisfying $\gamma(0) = p$ and $\gamma(1) = q$.
\end{definition}

\noindent
Critical points of $L(\gamma)$, if they exist, are called (length-minimizing) \textit{geodesics}.
We are interested in a particular manifold where such geodesics always exist.

\begin{definition}(Geodesic completeness)
A Riemannian manifold $(\mathcal M, g)$ is said to be geodesically complete if a geodesic exists between every two points $p,q \in \mathcal M$.
\end{definition}

\noindent
A major result connects geodesic completeness to metric completeness.
See \cite{MR3887684} for a proof.

\begin{definition}(Metric completeness)
A Riemannian manifold $(\mathcal M, g)$ is said to be metrically complete if $(\mathcal M, \operatorname{Dist})$ is a complete metric space.
\end{definition}

\begin{theorem}[Hopf--Rinow]
\label{th:hopf}
A connected Riemann manifold is metrically complete if and only if it is geodesically complete.
\end{theorem}

\noindent
Points in Euclidean space are updated along straight lines with constant velocity, $\gamma(t) := p + t \mu$.
Thus, setting $t = 1$ for simplicity, any point $q = \gamma(1)$ in a neighborhood of $p$ can parameterized by the associated velocity variable $\mu$.
The generalization of this parameterization to Riemannian manifolds, is given by the exponential map $\Exp \colon T\mathcal M \to \mathcal M$ \cite[Proposition~5.19]{MR3887684}.

\begin{definition}[Exponential map]
\label{th:exp map}
Let $(\mathcal M,g)$ be a geodesically complete Riemannian manifold.
Let $p \in \mathcal M$.
For any $\mu \in T_p \mathcal M$, there exists a unique geodesic $\gamma_\mu : [0,1] \to \mathcal M$, such that
\[
\gamma_\mu(0) = p
\quad\text{and}\quad
\dot{\gamma}_\mu(0) = \mu.
\]
The exponential map at $p$, denoted by $\Exp_p \colon T_p \mathcal M \to \mathcal M$, is defined as
\[
\Exp_p(\mu) := \gamma_\mu(1).
\]
\end{definition}

\noindent
The inverse of the exponential map $\Exp_p$ gives the velocity required to move from $p$ to $q$ in one unit of time.

\begin{definition}[Logarithmic map]
\label{def:log}
Let $(\mathcal M, g)$ be a complete Riemannian manifold.
Let $p \in \mathcal M$.
There exists an open neighborhood $U \subset T_p \mathcal M$ of the origin such that $\Exp_p: U \to V \subset \mathcal M$ is a diffeomorphism.
This inverse function is called the logarithmic map and denoted by $\operatorname{Log}_p \colon V \to U$.
\end{definition}

\noindent
Observe that $\operatorname{Log}_p(q) = q - p$ in Euclidean space, indicating that the velocity vector required to move from $p$ to $q$ is the negative of the vector required to move from $q$ to $p$; i.e., $\operatorname{Log}_p(q) = -\operatorname{Log}_q(p)$.
However, this relationship isn't well-defined in general because $\operatorname{Log}_p(q) \in T_p\mathcal{M}$ and $\operatorname{Log}_q(p) \in T_q\mathcal{M}$ belong to different sets.
Instead, the general relationship between $\operatorname{Log}_p(q)$ and $\operatorname{Log}_q(p)$ can be made precise via the concept of parallel transport.
See \cite[Chapters~4]{MR3887684} for more further details about parallel transport.

\begin{definition}[Parallel transport map]
\label{th:parallel}
Let $p,q \in \mathcal M$.
We consider a smooth curve $\gamma:[0,1] \to \mathcal M$ such that $\gamma(0) = p$ and $\gamma(1) = q$.
For a given tangent vector $\mu_p \in T_p \mathcal M$, there exists a unique parallel vector field $V$ along $\gamma$ such that $V(0)=\mu_p$.
We define the parallel transport map $P_{p\to q}\colon T_p \mathcal M \to T_q \mathcal M$ such that, 
\[ P_{p\to q}(\mu) = V(1). \]
\end{definition}

\noindent If $\gamma$ is a geodesic connecting $p$ to $q$, then the initial velocity $\operatorname{Log}_p(q)$ at time $t=0$ is parallel transported to the final velocity $-\operatorname{Log}_q(p)$ at time $t=1$; namely,
\begin{equation}
\label{eq:LogRelationship}
    -\operatorname{Log}_q(p) = P_{p\to q}\operatorname{Log}_p(q)
    .
\end{equation}
We now recall the following crucial result, which is the Riemannian manifold generalization of the Euclidean identity $\frac12 \frac{\partial}{\partial q}\|q - p\|_2^2 = q - p$.
A proof is provided in Appendix~\ref{sec:proof}.

\begin{lemma}
\label{th:grad distance}
Let $(\mathcal M,g)$ be a complete Riemannian manifold with $p \in \mathcal M$ fixed.
The following identity holds for every $q \in \mathcal M$ sufficiently close to $p$:
\begin{equation}
\label{eq:GeodesicEnergyIdentity}
    \operatorname{Grad}_q \bigg( \frac12 \operatorname{Dist}(q, p)^2 \bigg)
    =
    -\operatorname{Log}_{q}(p)
    \in T_q \mathcal M
    .
\end{equation}
\end{lemma}

\noindent
We leverage identity~\eqref{eq:GeodesicEnergyIdentity} in Section~\ref{sub:saddle_point_problem} to derive the method.
First, however, we require another fundamental property of parallel transport: along any smooth curve $\gamma$ in $\mathcal M$, it defines a linear isometry \cite[Proposition~5.5]{MR3887684}.
In particular, the angle between any two tangent vectors $\mu_p, \nu_p \in T_p \mathcal M$ is preserved,
\begin{equation}
\label{eq:Isometry}
    g_p(\mu_p,\nu_p) = g_q(\mu_q,\nu_q),
\end{equation}
where $\mu_q = P_{p\to q}(\mu_p)$ and $\nu_q = P_{p\to q}(\nu_p) \in T_q \mathcal M$ denote the parallel transports of $\mu_p$ and $\nu_p$, respectively.
Thus, upon defining $f(q) = \frac12 \operatorname{Dist}(q, p)^2$, we can invoke \eqref{eq:GradDefinition}--\eqref{eq:Isometry} to write
\begin{equation}
\label{eq:FirstVariationDistanceSquared}
    df_q(\mu_q)
    =
    g_q(\operatorname{Grad}_q f,\mu_q)
    =
    g_p(\operatorname{Log}_{p}(q), \mu_p)
    ,
    \quad
    \forall \mu_p \in T_p \mathcal{M}
    ,
\end{equation}
which is simply a generalization of the Euclidean directional derivative identity $\mu \cdot \frac{\partial}{\partial q} \frac12\|q - p\|_2^2 = (q-p)\cdot \mu$.

\subsection{The Stiefel manifold}

Let $n,m \in \mathbb N$ such that $n \ge m$.
The Stiefel manifold is defined as
\[ St(n,m) := \{ U \in \mathbb R^{n \times m} \mid U^{\mathsf T} U = I_m \}, \]
where $I_m$ is the $m \times m$ identity matrix.
In the sequel, we focus on the case $n = 3$ and $m =2$.

It is well-known that $St(n,m)$ is a smooth manifold \cite{zimmerman2022computing}.
The tangent space to $St(n,m)$ at $U \in St(n,m)$ is
\[ T_U St(n,m) = \{ W \in \mathbb R^{n \times m} \mid U^\tr W + W^\tr U = 0 \}. \]
For $U \in St(n,m)$, we define $\Pi_U \colon \mathbb R^{n \times m} \to T_U St(n,m)$ to be the orthogonal projection onto the tangent space to $St(n,m)$ at $U \in St(n,m)$.
For any $W \in \mathbb R^{n \times m}$, one has
\begin{equation}
\label{eq:projection}
\Pi_U(W) = W - U \sym(U^\tr W), 
\end{equation}
where $\sym$ extracts the symmetric part of a square matrix.

$St(n,m)$ is a complete Riemannian manifold when equipped with a certain family of Riemannian metrics $g$ \cite{knut2021lagrangian}.
In this work, we consider
\begin{equation}
\label{eq:Stiefel metric}
g_U(W, \widetilde{W}) := \operatorname{tr}( W^\tr \widetilde{W} ),
\end{equation}
for every $U \in St(n,m)$ and $W, \widetilde{W} \in T_U St(n,m)$.
This is known as the {Euclidean case} since $\operatorname{tr}( W^\tr \widetilde{W} )$ is the Frobenius inner product between $W$ and $\widetilde{W}$.

Since $St(n,m)$ is a complete Riemannian manifold when endowed with the Riemannian metric in~\eqref{eq:Stiefel metric}, Theorem~\ref{th:hopf} ensures that the exponential map is defined on all of $T \, St(n,m)$.
In fact, one has the explicit formula
\begin{equation}
\label{eq:exponential map}
\Exp_U(W) := \exp\left( W U^\tr - U W^\tr \right) U \exp\left( -U^\tr W \right),
\end{equation}
for all $U \in St(n,m)$ and $W \in T_U St(n,m)$ \cite{knut2021lagrangian}.
Here, $\exp(X) := \sum_{k=0}^\infty \frac1{k!} X^k$ is the matrix exponential for any square matrix $X \in \mathbb{R}^{\ell\times\ell}$.
\begin{equation}
\label{eq:Taylor expansion}
\Exp_U(\tau W) = U + \tau W - \frac{\tau^2}2 U W^\tr W  + O(\tau^3)
\end{equation}
for all $\tau > 0$, $U \in St(n,m)$, and $W \in T_U St(n,m)$; cf.\ \cite[Proposition~5.19]{MR3887684} and \cite[Equation~2.7]{MR1646856}.

\begin{lemma}
\label{th:Lipschitz}
There exists $K > 0$ such that for all $U \in St(3,2)$ and for all $W, \tilde{W} \in T_U St(3,2)$, one has
\[ \| D \Exp_U(W) - D \Exp_{U}(\tilde{W}) \| \le K | W - \tilde{W} |.  \]
\end{lemma}

A proof is provided in Appendix \ref{sec:proof 2}.

\section{Derivation}
\label{sec:continuous}
In the sequel, we use the notation $A \lesssim B$ to denote $A \le c B$, where $c > 0$ is a constant independent of $A$ and $B$ and all sequence indices such as $k$ and $h$.
As usual, $(\cdot,\cdot)$ denotes the $L^2(\Omega)$-inner product between scalar-, vector-, or matrix-valued functions.
We write as $g$ the Riemannian metric from \eqref{eq:Stiefel metric} on $St(3,2)$ and $\operatorname{Dist}$ the associated distance from Definition \ref{th:distance}.

\subsection{Continuous problem setting}
\label{sub:continuous_setting}
Let $\Gamma^D \subset \partial \Omega$ denote the relatively open portion of the boundary on which Dirichlet conditions are imposed.
For $\bm{x} \in \partial \Omega$, we denote by $\bm{n}(\bm{x}) \in \mathbb{R}^2$ the outward unit normal to $\partial \Omega$.

Let $V := H^2(\Omega)^3$.
Let $\bm{y}_D \in H^\frac12(\Gamma^D)^3$ and $\bm G_D \in H^\frac12(\Gamma^D)^3$.
Let $V_D := \{\bm w \in V \mid \bm w = \bm{y}_D \text{ and } (\nabla \bm w) \bm n = \bm G_D \text{ on } \Gamma^D \}$ be a subset of $V$.
We define the associated homogeneous space as $V_0 := \{\bm w \in V \mid \bm w = \bm 0 \text{ and } (\nabla \bm w) \bm n = \bm 0 \text{ on } \Gamma_D \}$.
We define the admissible set as 
\[ \mathbb A := \{\bm w \in V_D \mid \nabla \bm w \in St(3,2) \text{ a.e.\ in } \Omega \}. \]
Throughout the paper, we only consider settings where $\mathbb A \neq \emptyset$.

We now define the bilinear form
\[
    a(\bm v, \bm w) := \int_\Omega \nabla^2 \bm v : \nabla^2 \bm w \dd x ,
    \qquad \bm v, \bm w \in V.
    \]
    
The bending energy~\eqref{eq:EnergyFunctional} can therefore be rewritten as
\[ E(\bm{y}) = \frac12 a(\bm y, \bm y) - (\bm f, \bm y), \qquad \bm y \in V. \]

\begin{lemma}
\label{th:existence minimizer}
Assuming $\mathbb A \neq \emptyset$, there exists $\bm y \in \mathbb A$ such that
\begin{equation}
\label{eq:min problem}
E(\bm y) = \inf_{\bm v \in \mathbb A} E(\bm v).
\end{equation}
\end{lemma}
\noindent
We do not give the proof of this result since the proof of Lemma \ref{th:energy min}, below, is very similar.
At a sufficiently regular minimizer, one anticipates the existence of a Lagrange multiplier $\gamma \colon \Omega \to \mathbb{R}^{2\times 2}_{\sym}$ associated to the isometry constraint~\eqref{eq:IsometryConstraint}.
The corresponding KKT conditions would be
\begin{equation}
\label{eq:Euler-Lagrange}
\left\{ \begin{alignedat}{3}
a(\bm y, \bm w) + \int_\Omega \gamma : (\nabla \bm y^\tr \nabla \bm w + \nabla \bm w^\tr \nabla \bm y) \dd x &= (\bm f, \bm w), &&\quad \forall \bm w \in V_0, \\
\int_\Omega (\nabla \bm y^\tr \nabla \bm y - I):\zeta \dd x &= 0, &&\quad \forall \zeta \in \Upsilon,
\end{alignedat} \right.
\end{equation}
where $\Upsilon$ is an appropriate space of $2\times 2$ symmetric tensors.

\begin{remark}
One cannot ensure the uniqueness of the solutions to~\eqref{eq:min problem} since $\mathbb A$ is not convex.
\end{remark}

\begin{remark}
Even though Lemma~\ref{th:existence minimizer} shows that there exists a solution to~\eqref{eq:min problem}, we do not presently know how to prove the existence of a solution to \eqref{eq:Euler-Lagrange}.
This is due to the difficulty of proving an inf-sup condition on the second term in the left-hand of the first equation of \eqref{eq:Euler-Lagrange}.
For details, see, e.g., \cite{MR4839135}.
\end{remark}

\subsection{Riemannian proximal point algorithm}
This section proposes an iterative method to compute minimizers of \eqref{eq:Euler-Lagrange}.

Let $\tau > 0$ be a pseudo time-step.
Given $\bm{y}^0 \in \mathbb A$, the algorithm consists of computing
\begin{equation}
\label{eq:proximal 1}
\bm{y}^{k+1} \in \operatorname*{arg\,min}_{\bm{y} \in \mathbb A}\, E(\bm{y}) + \frac1{2\tau} \int_\Omega \operatorname{Dist}(\nabla \bm{y},\nabla \bm{y}^k)^2 \dd x =: \tilde{E}(\bm y),
\qquad k = 0,1,\ldots
\end{equation}

\begin{lemma}[Existence of minimizers]
\label{th:energy min}
Assume $\bm{y}^k \in \mathbb A$.
There exists a solution of \eqref{eq:proximal 1}, $\bm y^{k+1} \in \mathbb A$.
Moreover, one has
\begin{equation}
E(\bm y^{k+1}) \leq E(\bm y^k) -  \frac1{2 \tau} \int_\Omega \operatorname{Dist}(\nabla \bm y^{k+1}, \nabla \bm y^k)^2 \dd x.
\label{eq:EnergyDissipationContinuousLevel}
\end{equation}
\end{lemma}

\begin{proof}
Observe that $\tilde E(\bm y) \geq E(\bm y) > -\infty$ for all $\bm y \in \mathbb A$.
Thus, $m:=\inf_{\bm y \in \mathbb A} \tilde E(\bm y) > - \infty$ since $E$ is quadratic and coercive over $V$.
We now apply the direct method of the calculus of variations \cite{evans2022partial}.
Let $(\bm y _n)_{n \in \mathbb N} \in \mathbb A$ be a sequence such that, 
\[ \tilde E(\bm y_n) \mathop{\longrightarrow}_{n \to +\infty}  m. \]
Therefore, there exists $C > 0$, for all $n \in \mathbb N$, $\tilde E(\bm y_n) \le C$.
One deduces that for all $n \in \mathbb N$,
\[ \|\nabla^2 \bm y_n \|^2_{L^2(\Omega)} \le 2 \left(C + \| \bm f \|_{L^2(\Omega)} \| \bm y_n \|_{L^2(\Omega)} \right). \]
Applying \cite[Corollary~2.2]{bonito2021dg} successively to $\bm y_n$ and $\nabla \bm y_n$, one gets that
\[ \| \nabla^2 \bm y_n \|_{L^2(\Omega)} \lesssim 1. \]
Thus, there exists $\bm y^{k+1} \in V$ such that, up to a subsequence not relabeled,
$$
\bm y_n \to \bm y^{k+1} \text{ strongly in } H^1(\Omega)^3 \qquad \textrm{and} \qquad  
\nabla^2 \bm y_n \mathop{\rightharpoonup} \nabla^2 \bm y^{k+1} \text{ weakly in } L^2(\Omega)^{3\times 2 \times 2},
$$
as $n \to +\infty$.
$E$ is trivially convex and strongly continuous over $V$.
Therefore, using \cite[Corollary~3.9]{brezis}, $E$ is l.s.c.\ for the weak topology of $V$.
One thus has
\[ E(\bm y^{k+1}) \le \liminf_{n \to +\infty} E(\bm y_n). \]
Also, $\operatorname{Dist}(\nabla \bm y^k, \cdot)^2$ is easily seen to be continuous over $L^2(\Omega)^3$.
Therefore,
\[ \lim_{n \to +\infty} \operatorname{Dist}(\nabla \bm y_n, \nabla \bm y^k)^2 = \operatorname{Dist}(\nabla \bm y^{k+1}, \nabla \bm y^k)^2. \]
To summarize, one has
\[ \tilde E(\bm y^{k+1}) \le \liminf_{n \to +\infty} \tilde E(\bm y_n). \]
Thus, $\bm y^{k+1}$ is a minimizer of $\tilde E$.

Let us show the last statement of this lemma.
Since $\bm y^{k+1}$ is a minimizer of $\tilde E$, one has
\[ E(\bm y^{k+1}) + \frac1{2 \tau} \int_\Omega \operatorname{Dist}(\nabla \bm y^{k+1}, \nabla \bm y^k)^2 \dd x =\tilde E(\bm y^{k+1}) \le \tilde E(\bm y^k) = E(\bm y^k). \qedhere \]
\end{proof}

\subsection{Saddle-point problem}
\label{sub:saddle_point_problem}
To compute each new iterate $\bm y^{k+1}$, we aim to formally derive the Euler--Lagrange equation associated to~\eqref{eq:proximal 1}.
Viewing $\mathbb{A}$ as a Hilbert--Riemannian manifold \cite{biliotti2017riemannian} with tangent spaces $T_{\bm y}\mathbb{A} := \{ \bm w \in V_0 \mid \nabla \bm w \in T_{\nabla \bm y } St(3,2) \text{ a.e.\ in } \Omega \}$, we proceed by considering arbitrary variations $\bm w \in T_{\bm y^k}\mathbb{A}$ of the tangent vector along the geodesic connecting $\bm y^k$ to $\bm y^{k+1}$.
Using~\eqref{eq:FirstVariationDistanceSquared} to simplify the derivative of the squared distance term in~\eqref{eq:proximal 1}, we arrive at the following variational equation: Find $\bm y^{k+1} \in \mathbb{A}$ such that
\begin{equation}
\label{eq:partial continuous}
    a(\bm y^{k+1},\bm P_{\bm y^k \to \bm y^{k+1}}\bm w) + \frac1\tau\int_\Omega g_{\nabla \bm y^k} \big(\operatorname{Log}_{\nabla \bm y^k}(\nabla \bm y^{k+1}), \nabla \bm w\big) \dd x
    =
    (\bm f, \bm w),
    \quad
    \forall
    \bm w \in T_{\bm y^k}\mathbb{A}
    ,
\end{equation}
where $\bm P_{\bm y^k \to \bm y^{k+1}}$ denotes parallel transport along the geodesic from $\bm y^k$ and $\bm y^{k+1}$ and $\operatorname{Log}$ denotes the Riemannian logarithm in on the Stiefel manifold.
Since parallel transport is not readily computable on $\mathbb{A}$, we modify~\eqref{eq:partial continuous} by replacing $\bm P_{\bm y^k \to \bm y^{k+1}}$ with the identity.
This leaves the modified equation
\begin{equation}
\label{eq:partial continuous modified}
    a(\bm y^{k+1},\bm w) + \frac1\tau\int_\Omega g_{\nabla \bm y^k} \big(\operatorname{Log}_{\nabla \bm y^k}(\nabla \bm y^{k+1}), \nabla \bm w\big) \dd x
    =
    (\bm f, \bm w),
    \quad
    \forall
    \bm w \in T_{\bm y^k}\mathbb{A}
    ,
\end{equation}
which characterizes an inexact proximal point update, still denoted $\bm y^{k+1}$ for simplicity.

Although improved, the modified equation remains undesirable for Galerkin discretization for three reasons: (i) it still contains the nonlinear constraint $\bm y^{k+1} \in \mathbb{A}$; (ii), it involves test functions in the space $T_{\bm y^k}\mathbb{A}$, which changes at each iteration; and (iii) there is no explicit formula for the Riemmannian logarithm on the Stiefel manifold \cite{zimmerman2022computing}.
Fortunately, each issue can be avoided by introducing additional solution variables.

We now introduce the tangent variable $\mu^{k+1} = \tau^{-1} \operatorname{Log}_{\nabla \bm y^k}(\nabla \bm y^{k+1}) \in M^k$, where
\[ M^k := \left\{ \mu \in L^2(\Omega)^{3 \times 2} \mid \mu \in T_{\nabla \bm y^k} St(3,2) \text{ a.e.\ in } \Omega \right\}, \]
and note that $g_{\nabla \bm y^k} (\mu^{k+1}, \nabla \bm w) = \operatorname{tr}( (\nabla \bm w)^\tr \mu^{k+1} )$ by~\eqref{eq:Stiefel metric}.
One then inverts the definition of $\mu^{k+1}$, $\Exp_{\nabla \bm{y}^k} (\tau \mu^{k+1}) = \nabla \bm y^{k+1}$, leading to a system of equations that avoids the logarithm operator in~\eqref{eq:partial continuous modified}: Find $\bm y^{k+1} \in V_D$ and $\mu^{k+1} \in M^k$ such that
\begin{equation}
\label{eq:2var continuous}
    \begin{alignedat}{3}
      a(\bm{y}^{k+1},\bm w) + (\mu^{k+1},\nabla\bm w)
      &=
      (\bm f, \bm w),\quad
      &&\forall \bm{w} \in T_{\bm y^k}\mathbb{A}, \\
      (\nabla \bm{y}^{k+1}, v) - (\Exp_{\nabla \bm{y}^k} (\tau \mu^{k+1}),v)
      &= 0,
      &&\forall v \in L^2(\Omega)^{3\times2}.
    \end{alignedat}
\end{equation}
In this case, we note that the condition $y \in \mathbb{A}$ from~\eqref{eq:partial continuous} and~\eqref{eq:partial continuous modified} is now enforced via the second equation in~\eqref{eq:2var continuous}.

We have just dealt with issues (i) and (iii) above.
However, we are not done since~\eqref{eq:2var continuous} requires $\nabla \bm w \in T_{\nabla \bm y^k } St(3,2)$ and $\mu^{k+1} \in T_{\nabla \bm y^k} St(3,2)$ a.e.\ in $\Omega$, both of which are linear constraints that change at each iteration.
Weak imposition of the test function constraint is standard in the literature \cite[Remark~3.2]{bonito2021dg}, and involves introducing the trilinear form 
\begin{equation}
\label{eq:TrilinearForm}
    l(\nabla \bm y; \mu, \gamma) := \int_\Omega \gamma : \big( \nabla \bm y^\tr \mu + \mu^\tr \nabla \bm y \big) \dd x,  
\end{equation}
for all $\bm y \in V$, $\mu \in L^2(\Omega)^{3 \times 2}$ and $\gamma \in \Upsilon$, where $\Upsilon$ is an appropriate space of measurable functions $\gamma \colon \Omega \to \mathbb{R}^{2\times 2}_{\sym}$.
At this point, it is instructive to note that the definition of the exponential map in~\eqref{eq:exponential map} extends naturally to a smooth operator on matrices $W \in \mathbb{R}^{n\times m}$.
As such, we can also weakly enforce the condition $\mu^{k+1} \in M^k$ using the same trilinear form in~\eqref{eq:TrilinearForm}.

Upon introducing a multiplier variable $\gamma \in \Upsilon$, we now arrive at the final continuous problem in weak form: Find $(\bm y^{k+1}, \mu^{k+1}, \gamma^{k+1}) \in V_D \times L^2(\Omega)^{3 \times 2} \times \Upsilon$ such that  
\begin{subequations}
\label{eq:relaxed}
\begin{alignat}{3}
  a(\bm{y}^{k+1},\bm w) + (\mu^{k+1},\nabla\bm w) + l(\nabla \bm y^k;\nabla \bm{w}, \gamma^{k+1})
  &=
  (\bm f, \bm w),
  &&\quad \forall \bm{w} \in V_0, \\
  (\nabla \bm{y}^{k+1}, v) - (\Exp_{\nabla \bm{y}^k} (\tau \mu^{k+1}),v)
  &= 0,
  \label{eq:mid eq}
  &&\quad \forall v \in L^2(\Omega)^{3\times2},\\
  l(\nabla \bm{y}^k;\mu^{k+1}, \zeta) &= 0,
  &&\quad \forall \zeta \in \Upsilon.
  \label{eq:last eq}
 \end{alignat}
\end{subequations}
Note that, we have continued to use $\operatorname{Exp}_{\bm y^k}$ to denote the extension of the exponential map with arguments in $L^2(\Omega)^{3 \times 2}$ for simplicity.
Moreover, reflecting on~\eqref{eq:relaxed}, we observe that if $\mu^{k+1} = 0$, then $\bm y^{k+1} = \bm y^k \in \mathbb{A}$ and
\begin{subequations}
\label{eq:saddle Lag}
\begin{alignat}{3}
  a(\bm{y}^{k+1},\bm w) + l(\nabla \bm{y}^{k+1};\nabla \bm{w}, \gamma^{k+1})
  &=
  (\bm f, \bm w),
  &&\quad\forall \bm w \in V_0, \\
  (\nabla \bm y^{k+1})^\tr \nabla \bm y^{k+1}  &= I, &&~~ \text{ a.e. in } \Omega.
 \end{alignat}
\end{subequations}
Comparing~\eqref{eq:saddle Lag} to~\eqref{eq:Euler-Lagrange} reveals a fundamental relationship between fixed points of~\eqref{eq:relaxed} and saddle-points of~\eqref{eq:min problem}.
Recalling the energy dissipation property~\eqref{eq:EnergyDissipationContinuousLevel}, it is reasonable to expect that the proposed iterative method converges to saddle-points of the Lagrangian associated to minimizing $E$ over $\mathbb A$.
We reconsider on this observation in the next section, after the discretization is introduced.

\section{Discretization}
\label{sec:discrete}
Let $\left\{\mathcal{T}_h \right\}_{h>0}$ be a family of quasi-uniform and shape regular triangulations exactly fitting $\Omega \subset \mathbb R^2$; see \cite{ciarlet2002finite} for details.
The subdivision parameter $h$ in $\mathcal T_h$ indicates that $\max_{T \in \mathcal{T}_h} h_T \leq h$ where for a triangle $T \in \mathcal{T}_h$, $h_T := \mathrm{diam}(T)$ denotes its diameter.
For each $T \in \mathcal{T}_h$ and $\ell \in \mathbb{N}$, we use $\mathbb P^\ell(T)$ to denote the space of polynomials over $T$ of total degree at most $\ell$.
Local polynomial in this space are utilized to construct discontinuous piecewise polynomials belonging to the set $\mathbb P^\ell(\mathcal T_h) = \prod_{T \in \mathcal T_h} \mathbb P^\ell(T)$.

We choose to construct the approximate deformation $\bm y_h \approx \bm y$ within the non-conforming space $V_h := \mathbb P^2(\mathcal T_h)^3 \not\subset V$.
We use $M_h := \mathbb P^0(\mathcal T_h)^{3 \times 2}$ to approximate the tangent direction $\mu \in L^2(\Omega)^{3 \times 2}$ and $\Upsilon_h := \{\gamma_h \in \mathbb P^0(\mathcal T_h)^{2 \times 2} \mid \gamma_h^\tr = \gamma_h \}$ to approximate the Lagrange multiplier $\gamma \in \Upsilon$.

\subsection{Further notation}

We denote by $\mathcal{E}^i_h$ the set of all internal edges of $\mathcal{T}_h$ and by $\mathcal{E}_h^b$ the set of active	 boundary edges $e \in \mathcal{E}_h$ such that $e \subset \Gamma^D$.
The set of active inter-element boundaries is then defined as $\mathcal E_h := \mathcal E_h^b \cup \mathcal E_h^i$.
For an internal edge $e \in \mathcal E_h^i$, we set $\bm n_e$ to be one of the normals to $e$ in $\Omega$.
For an external edge $e \in \mathcal{E}_h^b$, we take $\bm n_e \equiv \bm n$ to be the outward unit normal to $e \subset \Gamma^D$.
We define the interior and boundary skeletons as
\[ \Gamma^i_h := \bigcup_{e \in \mathcal E_h^i} \{e\} \text{ and } \Gamma^b_h := \bigcup_{e \in \mathcal E_h^b} \{e\}. \]
The full skeleton $\Gamma_h$ is then defined as $\Gamma_h := \Gamma^i_h \cup \Gamma^b_h$.
Note that the boundary edges that are not part of $\Gamma^D$ are not included in $\Gamma_h$.

The broken gradient $\nabla_h$ is defined for a function $v \in \prod_{T \in \mathcal T_h} H^1(T)$ as $\nabla_h v = \sum_{T \in \mathcal{T}_h} \nabla (v{|_T})$.
The broken Hessian, written $\nabla^2_h$ is defined similarly.
For any internal edge $e \in \mathcal{E}^i_h$ and $\bm y_h \in V_h$, we define the jumps over $e$ as
\[ [\bm y_h]_e := \bm y_h^- - \bm y_h^+ \quad\text{ and }\quad [\nabla \bm y_h]_e := (\nabla_h \bm y_h^- - \nabla_h \bm y_h^+) \bm n_e , \]
where $\bm v_h^\pm(\bm x) = \lim_{s\to 0^+} \bm v_h(\bm x \pm s \bm n_e)$ for all $\bm x \in e$.
For $\bm y_h \in V_h$, we define $[\bm w_h]_e := \bm w_h$ and $[\nabla_h \bm y_h]_e := (\nabla_h \bm y_h) \bm n_e$.
We will not write the subscript $e$ whenever no confusion arises.
For $\bm w_h \in V_h$ and $\mu_h \in M_h$, we define the following discrete norm
\[ \|\bm w_h \|^2_{H^2_h(\Omega)} := \| \nabla^2_h \bm w_h \|_{L^2(\Omega)}^2 + \| h^{-\frac12} [\nabla_h \bm w_h] \|_{L^2(\Gamma_h)}^2 + \| h^{-\frac32} [\bm w_h] \|_{L^2(\Gamma_h)}^2. \]
The average of $\bm y_h \in V_h$ across $e \in \mathcal{E}^i_h$ is defined as
\[ \{\bm y_h\}_e(\bm x) := \frac12 \left( \bm y_h^-(x) + \bm y_h^+(x) \right), \quad \forall \bm x \in e. \]
For a boundary edge $e \in \mathcal E_h$, we set $\{ \bm y_h\}_e := \bm y_h$.

\subsection{Discrete energy}
The discrete admissible set is defined
\begin{equation}
\mathbb A_h :=  \{\bm w_h \in V_h \mid \nabla_h \bm w_h(\bm x_T) \in St(3,2), \ \forall T \in \mathcal T_h \}.
\end{equation}
As with the admissible set $\mathbb A$, we assume that $\mathbb A_h \neq \emptyset$.
Let $\bm y_h \in V_h$ and $\bm w_h \in V_h$.
The discrete bilinear form is defined by
\begin{multline}
a_h(\bm y_h, \bm w_h) := \int_\Omega \nabla^2_h \bm y_h : \nabla^2_h \bm w_h \dd x
     - \int_{\Gamma_h}  [\nabla_h \bm y_h] : \{\nabla^2_h \bm w_h\} \bm n \dd S - \int_{\Gamma_h}  [\nabla_h \bm w_h] : \{\nabla^2_h \bm y_h\} \bm n \dd S \\
    + \eta_1 \int_{\Gamma_h} h^{-1} [\nabla_h \bm y_h] : [\nabla_h \bm w_h] \dd S + \eta_0 \int_{\Gamma_h} h^{-3} [\bm y_h]\cdot [\bm w_h] \dd S,
\end{multline}
where $\eta_0,\eta_1 > 0$ are user-defined penalty parameters.
Note that Dirichlet boundary conditions are imposed weakly due to the non-conformity of the trial space $V_h \not \subset V$.
This leads us to define the associated linear form
\[ F_h(\bm w_h) := - \int_{\Gamma_h^b}  \bm G_D \cdot (\{\nabla^2_h \bm w_h\} \bm n) \bm n \dd S + \eta_1 \int_{\Gamma_h^b} h^{-1} \bm G_D \cdot (\nabla_h \bm w_h) \bm n \dd S + \eta_0 \int_{\Gamma_h^b} h^{-3} \bm y_D \cdot \bm w_h \dd S. \]
The discrete energy is then defined as
\begin{equation}
\label{eq:discrete energy}
    E_h(\bm y_h) := \frac12 a_h(\bm y_h, \bm y_h) - (\bm f, \bm y_h) - F_h(\bm y_h).
\end{equation}

\begin{lemma}
\label{th:existence discrete minimizer}
Assuming $\mathbb A_h \neq \emptyset$, there exists $\bm y_h \in \mathbb A_h$ such that
\begin{equation}
E_h(\bm y_h) = \inf_{\bm v_h \in \mathbb A_h} E_h(\bm v_h).
\end{equation}
\end{lemma}

\begin{proof}
By using \cite[Lemma~2.3]{bonito2021dg}, we have that $E_h$ is coercive in the sense $E_h(\bm y_h) \to +\infty$ if $\| \bm y_h \|_{H^2_h(\Omega)} \to +\infty$.
Noting that $\mathbb A_h$ is a closed subset of the finite-dimensional vector space $V_h$, we conclude that $E_h$ achieves a minimum over $\mathbb A_h$.
\end{proof}

\subsection{Proximal Galerkin method}
The above IPDG discretization of \eqref{eq:relaxed} delivers a discrete problem reading: for each $k = 0, 1, 2,\ldots$, find $(\bm y_h^{k+1}, \mu_h^{k+1}, \gamma_h^{k+1}) \in \mathbb A_h \times M_h \times \Upsilon_h$ such that
\begin{subequations}
\label{eq:discrete problem}
\begin{alignat}{3}
\label{eq:minimize}
  a_h(\bm{y}^{k+1}_h,\bm w_h) + (\mu^{k+1}_h,\nabla_h \bm w_h) + l_h(\nabla_h \bm y_h^k;\nabla_h \bm w_h, \gamma_h^{k+1})
  &=
  (\bm f, \bm w_h) + F_h(\bm w_h),
  \\
  \label{eq:stay in Stiefel}
  (\nabla_h \bm y_h^{k+1}, v_h) - \sum_{T \in \mathcal T_h} (\Exp_{\nabla_h \bm y_h^k(\bm x_T)} (\tau \mu_h^{k+1}),v_h)_{L^2(T)}
  &= 0,
  \\
  l_h(\nabla_h \bm y_h^k;\mu_h^{k+1}, \zeta_h) &= 0,
 \label{eq:stay tangent to Stiefel}
 \end{alignat}
\end{subequations}
for all $\bm w_h \in V_h$, $v_h \in M_h$, and $\zeta_h \in \Upsilon_h$, where the trilinear form
\[ l_h(\nabla_h \bm y_h;\mu_h, \gamma_h) := \sum_{T \in \mathcal T_h} \int_T \gamma_h : (\nabla_h \bm y_h^\tr \mu_h + \mu_h^\tr \nabla_h \bm y_h)(\bm x_T) \dd x,
\]
acts on any $\bm y_h \in V_h$, $\mu_h \in \mathbb P^1(\mathcal T_h)^{3 \times 2}$, and $\gamma_h \in \Upsilon_h$.

Our first result establishes the existence of solutions to~\eqref{eq:discrete problem}.
In Section~\ref{sub:discrete_energy_dissipation}, we use an energy dissipation property to show that the upper bound on the right-hand side of~\eqref{eq:stability} is, in fact, uniform in $k$ under mild assumptions on the initial iterate $\bm y_h^0$.

\begin{theorem}
\label{th:discrete solutions}
Let $C > 0$, independent of $k$, and $\bm y_h^k \in \mathbb A_h$.
For sufficiently small $\tau > 0$, independent of $k$, there exists a unique solution $(\bm y_h^{k+1}, \mu^{k+1}_h, \gamma_h^{k+1}) \in \mathbb A_h \times M_h \times \Upsilon_h$ to~\eqref{eq:discrete problem}.
There also exists $C > 0$, depending only on $h > 0$, such that
\[
    \| \bm y_h^{k+1} - \bm y_h^k \|_{H^2_h(\Omega)} + \| \mu^{k+1}_h \|_{L^2(\Omega)} + \| \gamma^{k+1}_h \|_{L^2(\Omega)}
    \le
    C \|\bm y_h^k\|_{H^2_h(\Omega)}
    .
\label{eq:stability}
\]

\end{theorem}

We postpone the proof of Theorem~\ref{th:discrete solutions} to Section~\ref{sub:proof_of_theorem_ref_th_discrete_solutions} since we need a few preliminary results for it.
Instead, we state the following corollary.

\begin{coro}
\label{coro}
Let $\bm y^k_h \in \mathbb A_h$.
Assume that the solution $(\bm y^{k+1}_h, \mu_h^{k+1}, \gamma_h^{k+1}) \in \mathbb A_h \times M_h \times \Upsilon_h$ of \eqref{eq:discrete problem} satisfies $\mu_h^{k+1} = 0$.
Then $\bm y^k_h = \bm y^{k+1}_h \in \mathbb A_h$ is a critical point of $E_h$ over $\mathbb A_h$.
\end{coro}

\begin{proof}
We note that $\bm y^k_h = \bm y_h^{k+1} \in \mathbb A_h$ by \eqref{eq:stay in Stiefel} and $\mu_h^{k+1} = 0$.
Using \eqref{eq:minimize} and $\mu_h^{k+1} = 0$, one also has
\[ a_h(\bm{y}^{k+1}_h,\bm w_h) + l_h(\nabla_h \bm y^{k+1}_h;\nabla_h \bm w_h, \gamma^{k+1}_h) = (\bm f, \bm w_h) + F_h(\bm w_h), \quad \forall \bm w_h \in V_h. \]
Thus, $(\bm y_h^{k+1}, \gamma_h^{k+1})$ is a critical point of the Lagrangian associated to minimizing $E_h$ over $\mathbb A_h$.
\end{proof}

\begin{remark}[Convergence criterion]
\label{rk:convergence}
Corollary \ref{coro} in conjunction with \eqref{eq:minimize} gives an obvious convergence criterion in the form of $\| \mu_h^{k+1} \|_{L^2(\Omega)} < \texttt{tol.}$, where $\texttt{tol.} > 0$ is a user-defined tolerance parameter.
\end{remark}

\subsection{Proof of Theorem~\ref{th:discrete solutions}}
\label{sub:proof_of_theorem_ref_th_discrete_solutions}

\begin{lemma}
\label{th:matrix lemma}
Let $A, B \in \mathbb R^{3 \times 2}$ such that $B$ has full rank.
One has
\[ |A^\tr B + B^\tr A| \ge 2 \sigma_2(B) |A|, \]
where $| \cdot |$ is the Frobenius norm on matrices and $\sigma_2(B) > 0$ is the smallest singular value of $B$.
\end{lemma}

\begin{proof}
Recall that $|A^\tr B + B^\tr A| = \sup_{|C| = 1} (A^\tr B + B^\tr A) : C$.
One has
\[ A^\tr B:C = \mathrm{tr}(B^\tr A C) = \mathrm{tr}(C A B^\tr) = C : B^\tr A, \]
if $C$ is symmetric.
Likewise, $A^\tr B:C = - C : B^\tr A$ if $C$ is skew-symmetric.
Therefore,
\[
    |A^\tr B + B^\tr A|
    =
    \sup_{|C| = 1} (A^\tr B + B^\tr A) : C
    =
    2\sup_{|C| = 1} C: B^\tr A
    =
    2 |B^\tr A|
    .
    \]
We now write the SVD of $B$.
There exists $U \in O_3(\mathbb R)$ and $V \in O_2(\mathbb R)$ such that $B = U \Sigma V^\tr$, with
\[ \Sigma = \begin{pmatrix}
\sigma_1 & 0 \\ 0 & \sigma_2 \\ 0 & 0
\end{pmatrix}, \]
where $\sigma_1 \ge \sigma_2 > 0$ are the singular values of $B$.
$\sigma_2 > 0$ since $B$ has full rank.
Thus,
\[ |B^\tr A| = |\Sigma^\tr U^\tr A| \ge \sigma_2 |U^\tr A| = \sigma_2 |A|, \]
since $U \in O_3(\mathbb R)$ and $V \in O_2(\mathbb R)$.
The result follows.
\end{proof}

\begin{proposition}
\label{th:inf-sup 1}
Let $\bm y_h^k \in \mathbb A_h$.
There exists $\alpha > 0$, such that
\[ \inf_{\mu_h \in M_h} \sup_{\zeta_h \in \Upsilon_h} \frac{l_h(\nabla_h \bm y_h^k; \mu_h, \zeta_h)}{\| \mu_h \|_{L^2(\Omega)} \|\zeta_h \|_{L^2(\Omega)}} \ge \alpha.  \]
\end{proposition}

\begin{proof}
Let $\mu_h \in M_h$.
By using the equality case in the Cauchy--Schwarz inequality, one has
\begin{equation*}
\begin{aligned}
\sup_{\zeta_h \in \Upsilon_h} \frac{ l_h(\nabla_h \bm y_h^k;\mu_h,\zeta_h)}{\|\zeta_h\|_{L^2(\Omega)}}
& = \left( \sum_{T \in \mathcal T_h} |T| |[(\nabla_h \bm y_h^k)^\tr \mu_h + \mu_h^\tr \nabla_h \bm y_h^k](\bm x_T)|^2 \right)^{\frac12}.
\end{aligned}
\end{equation*}
Let $T \in \mathcal T_h$.
We now apply Lemma \ref{th:matrix lemma} to get
\[ |[(\nabla_h \bm y_h^k)^\tr \mu_h + \mu_h^\tr \nabla_h \bm y_h^k](\bm x_T)| \ge 2 \sigma_2(\nabla_h \bm y_h^k(\bm x_T)) |\mu_h(\bm x_T)| = 2 |\mu_h(\bm x_T)|, \]
since $\bm y_h^k \in \mathbb A_h$.
Therefore, one has
\[ \sup_{\zeta_h \in \Upsilon_h} \frac{ l_h(\nabla_h \bm y_h^k;\mu_h,\zeta_h)}{\|\zeta_h\|_{L^2(\Omega)}} \ge 2 \| \mu_h \|_{L^2(\Omega)}.  
\qedhere
\]
\end{proof}

We recall the following result from \cite[Theorem~5.4]{MR4839135}.
\begin{proposition}
\label{th:inf-sup 2}
Let $\bm y_h^k \in \mathbb A_h$.
There exists $\beta > 0$, such that $\beta_h := \beta h_{\min}$ satisfies
\[ \inf_{\gamma_h \in \Upsilon_h} \sup_{\bm w_h \in V_h} \frac{l_h(\nabla_h \bm y_h^k; \nabla_h \bm w_h, \gamma_h)}{\| \bm w_h \|_{H^2_h(\Omega)} \|\gamma_h \|_{L^2(\Omega)}} \ge \beta_h,  \]
where $h_{\min} := \min_{T \in \mathcal T_h} h_T$, where $h_T$ is the size of the cell $T \in \mathcal T_h$.
\end{proposition}

Let $\bm y_h^k \in \mathbb A_h$ be fixed.
We define two continuous linear operators that will be instrumental to prove Theorem \ref{th:discrete solutions}.
Let $L_1: V_h \to \Upsilon_h$ and $L_2: M_h \to \Upsilon_h$
defined via the variational equations
\[ (L_1 (\bm w_h), \gamma_h ) = l_h(\nabla_h \bm y_h^k; \nabla_h \bm w_h, \gamma_h), \]
for all $\bm w_h \in V_h$ and $\gamma_h \in \Upsilon_h$, and 
\[ (L_2 (\mu_h), \zeta_h ) = l_h(\nabla_h \bm y_h^k; \mu_h, \zeta_h), \]
for all $\mu_h \in M_h$ and $\zeta_h \in \Upsilon_h$.
We need a final result before proving Theorem \ref{th:discrete solutions}.

\begin{lemma}
\label{th:inf-sup 3}
There exists $\beta' > 0$, such that $\beta'_h := \beta' h_{\min}$ satisfies
\[ \inf_{\mu_h \in \Ker(L_2)} \sup_{\bm w_h \in \Ker(L_1)} \frac{(\mu_h, \nabla_h \bm w_h)}{\|\mu_h \|_{L^2(\Omega)} \| \bm w_h \|_{H^2_h(\Omega)}} \ge \beta'_h.  \]
\end{lemma}

\begin{proof}
This proof follows parts of the proof of \cite[Theorem~5.4]{MR4839135}.
Let $\mu_h \in \Ker(L_2)$.
By definition of $L_2$, this entails that for all $T \in \mathcal T_h$,
\[ ((\nabla \bm y_h^k)^\tr \mu_h)(\bm x_T) + (\mu_h^\tr \nabla \bm y_h^k) (\bm x_T) = 0. \]
We define $\bm w_h \in V_h$ such that for all $T \in \mathcal T_h$, for all $x \in T$, $\bm w_h(x) := \mu_h(\bm x_T) (\bm x - \bm x_T)$.
One thus has 
\[ (\mu_h,\nabla_h \bm w_h) = \|\mu_h \|_{L^2(\Omega)} \|\nabla_h \bm w_h \|_{L^2(\Omega)}, \]
since we are in the equality case of the Cauchy--Schwarz inequality.
Also, for all $T \in \mathcal T_h$, one has
\[ 0 = ((\nabla \bm y_h^k)^\tr \mu_h)(\bm x_T) + (\mu_h^\tr \nabla \bm y_h^k) (\bm x_T) = ((\nabla \bm y_h^k)^\tr \nabla_h \bm w_h)(\bm x_T) + ((\nabla_h \bm w_h)^\tr \nabla \bm y_h^k) (\bm x_T), \]
and thus $\nabla_h \bm w_h \in \Ker(L_1)$.
Finally, since $\bm w_h$ is piecewise linear, one has
\[ \begin{aligned}
    \| \bm w_h \|_{H^2_h(\Omega)}^2 & = \sum_{e \in \mathcal E_h} \| h_e^{-\frac32} [\bm w_h] \|_{L^2(e)}^2 + \| h_e^{-\frac12} [\nabla_h \bm w_h] \|_{L^2(e)}^2
    \\
    & \lesssim 
    \sum_{T \in \mathcal T_h} h_T^{-4} \|\bm w_h \|_{L^2(T)}^2 + h_T^{-2} \|\nabla \bm w_h \|_{L^2(T)}^2
    \\
    & \lesssim 
    \sum_{T \in \mathcal T_h} h_T^{-2} \|\nabla \bm w_h \|_{L^2(T)}^2
    ,
\end{aligned} \]
where we have used a trace inequality and then a Poincar\'e inequality on each $T \in \mathcal T_h$ taking into account that $\bm w_h$ has a vanishing average over each $T \in \mathcal T_h$.
To summarize, there exists $\beta' > 0$, independent of $h > 0$, such that 
\[ \sup_{\bm w_h \in \Ker(L_1)} \frac{(\mu_h, \nabla_h \bm w_h)}{\| \bm w_h \|_{H^2_h(\Omega)}} \ge \beta' h_{\min} \|\mu_h \|_{L^2(\Omega)},  \]
which finishes the proof.
\end{proof}

\begin{remark}[$h$-uniform inf-sup constant]
Let $\mu_h \in M_h$.
Following, e.g., \cite[Chapter~5]{di2010discrete}, one can define the following discrete norm:
\[\| \mu_h \|_{\DG}^2 := \sum_{e \in \mathcal{E}_h} h_e^{-1}\|[\mu_h]\|_{L^2(e)}^2, \]
which is a discrete approximation of the $H^1(\Omega)$-norm on $M_h$.
Considering the associated dual norm, denoted $\| \cdot \|_{DG^\prime}$, it is possible to prove a uniform version of Lemma \ref{th:inf-sup 3}: there exists $\beta' > 0$, independent of $h$, such that
\[ \inf_{\mu_h \in \Ker(L_2)} \sup_{\bm w_h \in \Ker(L_1)} \frac{(\mu_h, \nabla_h \bm w_h)}{\|\mu_h \|_{\DG^\prime} \| \bm w_h \|_{H^2_h(\Omega)}} \ge \beta'.  \]
We avoid incorporating this norm into~Theorem~\ref{th:discrete solutions} as the inf-sup constant $\beta_h$ in Proposition~\ref{th:inf-sup 2} also depends on $h$.
Using $\| \cdot \|_{\DG^\prime}$ also complicates the proof and is unnecessary for establishing the bounds given in Corollary~\ref{cor:ImprovedBounds}, below.
\end{remark}

We are now in a position to prove Theorem \ref{th:discrete solutions}.

\begin{proof}[Proof of Theorem \ref{th:discrete solutions}]
Our goal is to apply \cite[Theorem~2.1]{caloz1997numerical}.
Let $X_h := V_h \times M_h \times \Upsilon_h$.
Let $z_h = (\bm w_h, v_h, \zeta_h) \in X_h$, we endow $X_h$ with the norm $\| z_h \|_{X_h} := \| \bm w_h \|_{H^2_h(\Omega)} + \| v_h \|_{L^2(\Omega)} + \| \zeta_h \|_{L^2(\Omega)}$.
Defined as such, $X_h$ is a finite-dimensional Banach space.
We also define $Y_h := V_h \times \Ker(L_2) \times \Upsilon_h \subset X_h$, which is also a Banach space when endowed with the norm of $X_h$.

Let $u_h = (\bm y_h, \mu_h, \gamma_h) \in Y_h$.
Let $A_h(u_h) : Y_h \to X_h'$ be the nonlinear operator defined,  for all $z_h = (\bm w_h, v_h, \zeta_h) \in X_h$, by
\begin{multline*}
\left<A_h(u_h), z_h \right>_{X_h',X_h} := a_h(\bm{y}_h,\bm w_h) + (\mu_h,\nabla_h \bm w_h) + l_h(\nabla_h \bm y_h^k;\nabla_h \bm w_h, \gamma_h) - (\bm f, \bm w_h) - F_h(\bm w_h)  \\
   + (\nabla_h \bm y_h, v_h) -  \sum_{T \in \mathcal T_h} (\Exp_{\nabla_h \bm y_h^k(\bm x_T)} (\tau \mu_h),v_h)_{L^2(T)}.
\end{multline*}

The exponential map~\eqref{eq:exponential map} is smooth by Lemma~\ref{th:Lipschitz}.
In turn, $A_h$ is a $\mathcal C^1$ mapping since the remaining terms in $A_h$ are linear.
We do not compute the exact expression of the Fr\'echet derivative $D \Exp_{\nabla_h \bm y_h^k(\bm x_T)}$ since it is not needed for the rest of this proof.
The primary result we need is
\begin{equation}
\label{eq:ExpAt0}
    D \Exp_{\nabla_h \bm y_h^k(\bm x_T)}(0)(\tilde{\mu}_h) = \tilde{\mu}_h, \quad \forall \tilde{\mu}_h \in M_h.
\end{equation}

Let $u_0 := (\bm y_h^k, 0, 0) \in Y_h$.
It is straightforward to show that
\[ \|A_h(u_0) \|_{X_h'} \le \|a_h\| \| \bm y_h^k \|_{H^2_h(\Omega)} + \| \bm f \|_{L^2(\Omega)} + C'(\| \bm y_D \|_{L^2(\Gamma_D)} + \| \bm G_D \|_{L^2(\Gamma_D)}) =: \varepsilon, \]
where $C' > 0$ is a constant independent of $h > 0$.
Note that $\varepsilon > 0$ is linearly dependent on $\| \bm y_h^k \|_{H^2_h(\Omega)}$.
We now seek to show that $DA_h(u_0)$ is bijective, which is equivalent to showing that the following linearized subproblem is well-posed: find $(\bm y_h, \mu_h, \gamma_h) \in Y_h$ such that
\begin{equation}
\label{eq:linear equation}
\begin{aligned}
  a_h(\bm{y}_h,\bm w_h) + (\mu_h,\nabla_h \bm w_h) + l_h(\nabla_h \bm y_h^k;\nabla_h \bm w_h, \gamma_h)
  &=
  f_1(\bm w_h),
  &&~~ \forall \bm w_h \in V_h, \\
  (\nabla_h \bm y_h, v_h) - \tau(\mu_h,v_h) 
  &= f_2(v_h),
  &&~~ \forall v_h \in M_h,\\
  l_h(\nabla_h \bm y_h^k;\mu_h, \zeta_h) &= 0, &&~~ \forall \zeta_h \in \Upsilon_h,
 \end{aligned}
\end{equation}
where $f_1: V_h \to \mathbb R$ and $f_2: M_h \to \mathbb R$ are continuous linear forms that represent a general right-hand side.
Indeed, the first two equations of~\eqref{eq:linear equation} arise from computing $DA_h(u_0)$ over $X_h$ and applying~\eqref{eq:ExpAt0}, while the condition $\mu_h \in \Ker(L_2)$ can be verified by testing the third equation with the element indicator function $\chi_T(x) := \begin{cases}
    1 & \text{ if } x \in T,\\
    0 & \text{ otherwise, }
\end{cases}$ in each component of $\zeta_h$.

We now seek to prove that \eqref{eq:linear equation} admits a unique solution.
Our goal is to use \cite[Theorem~3.1]{nicolaides1982existence} to get that result.
Propositions \ref{th:inf-sup 1} and \ref{th:inf-sup 2} gives us two of the three conditions required to apply \cite[Theorem~3.1]{nicolaides1982existence}.
The remaining condition consists of showing that there exists a unique solution to the subproblem of finding $(\bm y_h, \mu_h) \in \Ker(L_1) \times \Ker(L_2)$ such that
\begin{equation}
\label{eq:perturbed}
\begin{aligned}
a_h(\bm y_h, \bm w_h) + (\mu_h, \nabla_h \bm w_h) &= f_1(\bm w_h), &&~~ \forall \bm w_h \in \Ker(L_1), \\
(\nabla_h \bm y_h, v_h) - \tau (\mu_h, v_h) &= f_2(v_h), &&~~ \forall v_h \in \Ker(L_2).
\end{aligned}
\end{equation}

We aim to check the conditions of \cite[Theorem 4.3.2]{MR3097958} in order to show that there exists a unique solution to \eqref{eq:perturbed}.
Recalling \cite[Lemma~2.3]{bonito2021dg}, there exists $\kappa > 0$, independent of $h > 0$, such that
\[ a_h(\bm w_h, \bm w_h) \ge \kappa \| \bm w_h \|_{H^2_h(\Omega)}^2, \quad \forall \bm w_h \in V_h. \]
Thus $a_h$ is coercive over $\mathrm{Ker}(L_1) \subset V_h$.
Lemma \ref{th:inf-sup 3} is then the last missing piece of the puzzle, which gives the LBB condition in \cite[Theorem 4.3.2]{MR3097958}.
We can thus apply \cite[Theorem 4.3.2]{MR3097958} to \eqref{eq:perturbed} and get
\begin{equation}
\label{eq:bounds}
\begin{aligned}
\| \bm y_h \|_{H^2_h(\Omega)} & \le \frac{(\beta'_h)^2 + 4\tau \|a_h\| }{\kappa (\beta'_h)^2} \|f_1\|_{V_h'} + \frac{2 \|a_h \|^{\frac12}}{\sqrt{\kappa} \beta'_h} \|f_2 \|_{M_h'}, \\
\| \mu_h \|_{L^2(\Omega)} & \le \frac{2 \|a_h \|^{\frac12}}{\sqrt{\kappa} \beta_h'} \|f_1\|_{V_h'} + \frac{4\|a_h\|}{\tau \|a_h \| + 2 (\beta'_h)^2} \|f_2 \|_{M_h'}.
\end{aligned}
\end{equation}
Note that the bounds in \eqref{eq:bounds} do not blow up when $\tau \to 0$.
To finish, we apply \cite[Theorem~3.1]{nicolaides1982existence} and thus $DA_h(u_0):Y_h \to X_h'$ is bijective onto $X_h'$.
We also deduce from \cite[Theorem~3.1]{nicolaides1982existence}  that
\[ \| DA_h(u_0)^{-1} \|_{\mathcal L(X_h',Y_h)} \le C(\tau), \]
where $C(\tau)$ is a function which remains bounded when $\tau \to 0$.

We now return to our goal of applying \cite[Theorem~2.1]{caloz1997numerical}.
We consider $u_h \in Y_h$ such that $\| u_h - u_0 \|_{Y_h} \le \delta$ where $\delta > 0$.
Observe that
\[ \begin{aligned}
\|D A_h(u_h) - D A_h(u_0)\|_{\mathcal L(Y_h,X_h')} & \le \left\|\sum_{T \in \mathcal T_h} D \Exp_{\nabla_h \bm y_h^k(\bm x_T)}(\tau \mu_h) - D\Exp_{\nabla_h \bm y_h^k(\bm x_T)}(0) \right\|_{\mathcal L(Y_h,X_h')} \\
& \le K \tau \| \mu_h \|_{L^2(\Omega)} \le K \tau \delta,
\end{aligned} \]
where we have used Lemma \ref{th:Lipschitz} to get a Lipschitz constant $K>0$ independent of $k$ and $h$.
Note that the upper bound above becomes arbitrarily small as $\tau \to 0$.
Also note that having $u_h \in Y_h$ is essential to applying Lemma \ref{th:Lipschitz} since one needs $\mu_h(\bm x_T) \in T_{\nabla \bm y_h^k (\bm x_T)} St(3,2)$ for all $T \in \mathcal T_h$.

Having now checked the conditions of \cite[Theorem~2.1]{caloz1997numerical}, we deduce the existence of a unique solution $(\bm y_h^{k+1}, \mu_h^{k+1},  \gamma_h^{k+1}) \in Y_h$ to \eqref{eq:discrete problem} for all $\tau > 0$ small enough.

We finally test \eqref{eq:stay in Stiefel} with $\chi_T$ in each component of $v_h$ to arrive at
\[ \nabla_h \bm y_h^{k+1}(\bm x_T) = \Exp_{\nabla_h \bm y_h^k(\bm x_T)}(\tau \mu_h^{k+1}(\bm x_T)) \in St(3,2). \]
Testing \eqref{eq:stay tangent to Stiefel} with $\chi_T$ in each component of $\zeta_h$, one gets $\mu_h^{k+1}(\bm x_T) \in T_{\nabla_h \bm y_h^k(\bm x_T)} St(3,2)$.
Combining these last two results, we deduce that $\bm y_h^{k+1} \in \mathbb A_h$.
\end{proof}

\subsection{Discrete energy dissipation}
\label{sub:discrete_energy_dissipation}

The following energy dissipation property allows us to show that there is a $k$-uniform upper bound on the right-hand side of~\eqref{eq:stability}.

\begin{proposition}[Discrete energy dissipation]
\label{th:energy decrease}
Fix $\bm y_h^k \in \mathbb A_h$ and let $(\bm y_h^{k+1}, \mu_h^{k+1}, \gamma_h^{k+1}) \in \mathbb A_h \times M_h \times \Upsilon_h$ be a corresponding solution of~\eqref{eq:discrete problem}.
Assume $\bm y_h^k \neq \bm y_h^{k+1}$.
For $\tau > 0$ small enough, one has
\[ E_h(\bm y_h^{k+1}) < E_h(\bm y_h^k). \]
\end{proposition}

\begin{proof}
Let us define $\delta \bm y_h := \bm y_h^{k+1} - \bm y_h^k$.
One has
\[  \begin{aligned}
E_h(\bm y_h^{k+1}) - E_h(\bm y_h^k) &= \frac12a_h(\bm y_h^{k+1}, \delta \bm y_h) - (\bm f, \delta \bm y_h) + \frac12 a_h(\delta \bm y_h, \bm y_h^k), \\
& \le a_h(\bm y_h^{k+1}, \delta \bm y_h) - (\bm f, \delta \bm y_h),
\end{aligned} \]
since 
\[ a_h(\delta \bm y_h, \bm y_h^k) = a_h(\delta \bm y_h, \bm y_h^k) - a_h(\delta \bm y_h, \delta \bm y_h) \le a_h(\bm y_h^{k+1}, \bm y_h^k), \]
because $a_h$ is coercive.
Note that the inequality above is strict if $\bm y_h^k \neq \bm y_h^{k+1}$.
Using \eqref{eq:minimize}, one has
\[ E_h(\bm y_h^{k+1}) - E_h(\bm y_h^k) \le - (\mu_h^{k+1}, \nabla_h \delta \bm y_h) - l_h(\nabla_h \bm y_h^k; \nabla_h \delta \bm y_h, \gamma_h^{k+1}). \]
Recalling \eqref{eq:Taylor expansion}, there exists $C > 0$, independent of $T \in \mathcal T_h$, such that
one has
\begin{equation}
\label{eq:Taylor 2}
| \nabla_h \delta \bm y_h(\bm x_T)  - \tau \mu_h^{k+1}(\bm x_T) | \le C \tau^2, \quad \forall T \in \mathcal T_h.
\end{equation}
The constant $C$ above is independent from $T \in \mathcal T_h$ because in the second order term in \eqref{eq:Taylor expansion}, $U \in St(3,2)$ and $St(3,2)$ is a bounded set.
Therefore, $- (\mu_h^{k+1}, \nabla_h \delta \bm y_h) \le 0$, for $\tau > 0$ small enough.

Finally, using \eqref{eq:Taylor 2}, one has
\[ l_h(\nabla_h \bm y_h^k; \nabla_h \delta \bm y_h, \gamma_h^{k+1}) = \tau l_h(\nabla_h \bm y_h^k; \mu_h^{k+1}, \gamma_h^{k+1}) + O(\tau^2) = O(\tau^2), \]
because of \eqref{eq:stay tangent to Stiefel}, which finishes the proof.
\end{proof}

\begin{coro}[{k}-uniform bound]
\label{cor:ImprovedBounds}
    If $\|\bm y_h^{0}\|_{H^2_h(\Omega)}\lesssim 1$, then there exists a sequence of solutions $(\bm y_h^k, \mu_h^k)_k$ of \eqref{eq:discrete problem} such that
    \[
        \|\bm y_h^k\|_{H^2_h(\Omega)} \lesssim 1,
    \]
    for all sufficiently small $\tau > 0$.
\end{coro}
\begin{proof}
We apply alternatively Theorem \ref{th:discrete solutions} and Proposition \ref{th:energy decrease} in order to establish the existence of successive iterates of decreasing energy.
Recalling \cite[Lemma~2.3]{bonito2021dg}, there exists $\kappa > 0$, independent of $h > 0$, such that
\[ a_h(\bm w_h, w_h) \ge \kappa \| \bm w_h \|_{H^2_h(\Omega)}^2. \]
Therefore, for each $k \in \mathbb N$, using Proposition \ref{th:energy decrease}, one has
    \[
        \frac{\kappa}{2} \|\bm y_h^k\|_{H^2_h(\Omega)}^2
        \leq
        a_h(y_h^k,y_h^k)
        \leq
        E_h(\bm y_h^{0})
        +
        \|f\|_{L^2(\Omega)}\|\bm y_h^{k}\|_{L^2(\Omega)}
        +
        \|F_h\|_{V_h^\prime}\|\bm y_h^k\|_{H^2_h(\Omega)}
        .
    \]
By applying Young's inequality and incorporating the $\| \bm y_h^k \|_{H^2_h(\Omega)}^2$ in the right-hand side into the left-hand side, one sees that $\| \bm y_h^k \|_{H^2_h(\Omega)} \lesssim 1$, as necessary.
\end{proof}

\section{Convergence}
\label{sec:convergence}
Our goal is to show that sequences of solutions of the discrete problem converge to solutions of \eqref{eq:min problem}.
To do so, our tools are compactness and $\Gamma$-convergence.
Before stating the $\Gamma$-convergence result in Section~\ref{sub:discrete_hessian}, we introduce the concept of discrete Hessian, which is central to the proof.

\subsection{Compactness}
This result is very close to \cite[Proposition~5.1]{bonito2021dg} and \cite[Theorem~9(i)]{MR4839135}.

\begin{proposition}[Compactness]
\label{th:compactness}
Let $C > 0$ and $(\bm y_h)_{ h > 0} \in \mathbb A_h$ be a sequence such that for all $h > 0$, $E_h(\bm y_h) \le C$.
There exists $\bm y \in \mathbb A$ such that, up to a subsequence (not relabeled),
\[ \begin{aligned}
&\bm y_h \to \bm y \text{ strongly in } L^2(\Omega)^3, \\
& \nabla_h \bm y_h \to \nabla \bm y \text{ strongly in } L^2(\Omega)^{3 \times 2}, \\
\end{aligned} \]
\end{proposition}

\begin{proof}
The proof follows the same first three steps as the proof of \cite[Proposition~5.1]{bonito2021dg} which are not repeated for concision.
From these three steps, one deduces that there exists $\bm y \in V_D$ such that $\bm y_h \to \bm y$ strongly in $L^2(\Omega)$ and $\nabla_h \bm y_h \to \nabla \bm y$ strongly in $L^2(\Omega)$.
The only thing that remains to be proved is that $\bm y \in \mathbb A$ i.e.\ $\nabla \bm y \in St(3,2)$ a.e.\ in $\Omega$.

On the one hand, one has
\[ \begin{aligned}
\| \nabla_h \bm y_h^\tr \nabla_h \bm y_h - \nabla \bm y^\tr \nabla \bm y \|_{L^1(\Omega)} &= \| \nabla_h \bm y_h^\tr (\nabla_h \bm y_h - \nabla \bm y) + (\nabla_h \bm y_h - \nabla \bm y)^\tr \nabla \bm y \|_{L^1(\Omega)}, \\
& \le \left( \| \nabla_h \bm y_h \|_{L^2(\Omega)} + \| \nabla \bm y \|_{L^2(\Omega)} \right) \|\nabla_h \bm y_h - \nabla \bm y\|_{L^2(\Omega)} \to 0,
\end{aligned} \]
since $(\bm y_h)_{h > 0}$ and $(\nabla_h \bm y_h)_{h > 0}$ are convergent and thus bounded in $L^2(\Omega)$ and $\nabla_h \bm y_h \to \nabla \bm y$ strongly in $L^2(\Omega)$.
On the other hand, since $\nabla_h \bm y_h(\bm x_T) \in St(3,2)$, one also has
\[ \begin{aligned}
\| \nabla_h \bm y_h^\tr \nabla_h \bm y_h - I \|_{L^1(T)} &= \| \nabla_h \bm y_h^\tr \nabla_h \bm y_h - \nabla_h \bm y_h(\bm x_T)^\tr \nabla_h \bm y_h (\bm x_T) \|_{L^1(T)}, \\
& \le \left( \| \nabla_h \bm y_h \|_{L^2(T)} + \| \nabla_h \bm y_h (\bm x_T) \|_{L^2(T)} \right) \|\nabla_h \bm y_h - \nabla_h \bm y_h(\bm x_T)\|_{L^2(T)},
\end{aligned} \]
for each $T \in \mathcal T_h$.
Using Jensen's inequality, we conclude that
\[ \| \nabla_h \bm y_h (\bm x_T) \|_{L^2(T)} \le \| \nabla_h \bm y_h \|_{L^2(T)}. \]
Next, note that $\nabla_h \bm y_h(\bm x_T) = \frac{1}{|T|}\int_T \nabla_h \bm y_h \dd x$ since $\nabla_h \bm y_h$ is a $\mathbb P^1$ function on $T$.
Thus, Poincar\'e's inequality implies that
\[ \| \nabla_h \bm y_h -\nabla_h \bm y_h(\bm x_T)\|_{L^2(T)} \le h_T \| \nabla^2_h \bm y_h \|_{L^2(T)}. \]
In turn, we conclude that
\[ \| \nabla_h \bm y_h^\tr \nabla_h \bm y_h - I \|_{L^1(T)} \lesssim h_T \| \nabla^2_h \bm y_h \|_{L^2(T)} \| \nabla_h \bm y_h \|_{L^2(T)}. \]
Summing over all $T \in \mathcal T_h$ and using that $\| \nabla^2_h \bm y_h \|_{L^2(\Omega)} \lesssim 1$ and $\| \nabla_h \bm y_h \|_{L^2(\Omega)} \lesssim 1$, one finds that
\[ \| \nabla_h \bm y_h^\tr \nabla_h \bm y_h - I \|_{L^1(\Omega)} \lesssim h \to 0. \]
Finally, the triangle inequality implies that $\| \nabla \bm y^\tr \nabla \bm y - I \|_{L^1(\Omega)} = 0$, which completes the proof.  
\end{proof}

\subsection{Discrete Hessian}
\label{sub:discrete_hessian}

For $\bm y_h \in V_h$, the broken Hessian $\nabla_h^2 \bm y_h$ does not converge to $\nabla^2 \bm y$ when $\bm y_h \to \bm y$ in $L^2(\Omega)^3$.  
Instead, the discrete Hessian $H_h(\bm y_h)$ \cite{pryer2012discontinuous} restores convergence by incorporating jumps of both the function and its gradient \cite{di2010discrete}.  
We briefly outline the construction of $H_h$ and its key properties.  

For each $e \in \mathcal{E}_h$, we can define a local lifting operator $r_e:L^2(e)^{3\times 2} \to \mathbb P^1(\mathcal T_h)^{3\times 2\times 2}$ via the variational equation
\begin{equation}\label{e:re}
\int_\Omega r_e(v) : \tau_h \dd x = \int_e  \{\tau_h\}_e \bm n_e : v \dd S, \quad \forall \tau_h \in \mathbb P^1(\mathcal T_h)^{3\times 2\times 2}. 
\end{equation}
The support of $r_e$ is denoted by $\omega(e)$ and consists of two elements of $\mathcal T_h$ when $e \in \mathcal{E}_h^i$ and one element of $\mathcal T_h$ when $e$ is a boundary edge.
The global lifting operator $R_h: L^2(\Gamma_h^i \cup \Gamma^D)^{3\times 2} \to \mathbb P^1(\mathcal T_h)^{3\times 2\times 2}$ is defined for $v \in L^2(\Gamma_h^i \cup \Gamma^D)^{3\times 2}$ by
\[ R_h(v):= \sum_{e \in \mathcal{E}_h^i \cup \mathcal E_h^D} r_e(v).  \]
We define the discrete Hessian of $\bm w_h \in V_h$ as follows:
\[ H_h(\bm w_h) := \nabla^2_h \bm w_h - R_h([\nabla_h \bm w_h]), \]
where $[\nabla_h \bm w_h] \in L^2(\Gamma_h)^{3 \times 2}$ denotes the jumps of the broken gradient defined on the skeleton of the mesh $\Gamma_h$.

The following lemmata are restated from \cite[Proposition~4.3]{bonito2021dg} and \cite[Proposition~4.5]{bonito2021dg}.
\begin{lemma}
\label{th:weak hessian}
Let  $\bm y \in V$ and $(\bm y_h)_{h > 0} \in V_h$ such that $\bm y_h \to \bm y$ strongly in $L^2(\Omega)$.
Up to a subsequence (not relabeled), one has
\[ H_h(\bm y_h) \rightharpoonup \nabla^2 \bm y \text{ weakly  in } L^2(\Omega). \]
\end{lemma}	

\begin{lemma}
\label{th:lower bound}
Let $\bm y_h \in V_h$.
For sufficiently large penalty parameters $\eta_0,\eta_1 > 0$, one has
\[ \frac12 \| H_h(\bm y_h) \|^2_{L^2(\Omega)} - (\bm f, \bm y_h) \le E_h(\bm y_h). \]
\end{lemma}

\subsection{$\Gamma$-convergence}

\begin{proposition}[$\liminf$]
\label{th:liminf}
Let $C>0$ and $(\bm y_h)_{h>0} \in \mathbb{A}_h$ such that $E_h(\bm y_h) \le C$, and $\bm y \in \mathbb A$.
We assume that, as $h \to 0$,
\[ \left\{ \begin{aligned}
& \bm y_h \to \bm y \text{ strongly in } L^2(\Omega)^3, \\
& \nabla_h \bm y_h \to \nabla \bm y \text{ strongly in } L^2(\Omega)^{3 \times 2}. \\
\end{aligned} \right.\]
Then, for sufficiently large $\eta_0, \eta_1 > 0$, it holds that
\[ E(\bm y) \le \liminf_{h \to 0} E_h(\bm y_h). \]
\end{proposition}

\begin{proof}
Invoking Lemma \ref{th:lower bound}, one has
\[ \frac12 \| H_h(\bm y_h) \|^2_{L^2(\Omega)} - (\bm f, \bm y_h) \le E_h(\bm y_h). \]
Lemma \ref{th:weak hessian} entails that
\[ \| \nabla^2 \bm y \|_{L^2(\Omega)} \le \liminf_{h \to 0} \| H_h(\bm y_h) \|_{L^2(\Omega)}. \]
Together with $\bm y_h \to \bm y$ strongly in $L^2(\Omega)$, one gets
\[ E(\bm y) = \frac12 \|\nabla^2 \bm y\|^2_{L^2(\Omega)} - (\bm f, \bm y) \le \liminf_{h \to 0} E_h(\bm y_h). \qedhere \] 
\end{proof}

\begin{proposition}[$\limsup$]
\label{th:limsup}
Let $\bm y \in \mathbb A$.
There exists $(\bm y_h)_{h > 0} \in \mathbb A_h$ such that $\bm y_h \to \bm y$ strongly in $L^2(\Omega)$ and $E(\bm y) \ge \limsup_{h \to 0} E_h(\bm y_h)$.
\end{proposition}
\noindent
The proof of Proposition~\ref{th:limsup} is the same as the proof of \cite[Theorem~9~(iii)]{MR4839135}.
It is omitted for concision.

\subsection{Convergence of discrete minimizers}

Note that Lemma \ref{th:existence discrete minimizer} ensures that our method can produce the minimizers invoked in the following theorem.

\begin{theorem}
Let $(\bm y_h)_{h > 0} \in \mathbb A_h$ be a sequence of global minimizers of $E_h$ over $\mathbb A_h$.
For sufficiently large $\eta_0, \eta_1 > 0$, there exists $\bm y \in \mathbb A$, global minimizer of $E$ over $\mathbb A$ such that, up to a subsequence (not relabeled),
\[ \begin{aligned}
&\bm y_h \to \bm y \text{ strongly in } L^2(\Omega), \\
&\nabla_h \bm y_h \to \nabla \bm y \text{ strongly in } L^2(\Omega), \\
&\lim_{h \to 0} E_h(\bm y_h) = E(\bm y).
\end{aligned} \]
\end{theorem}

\begin{proof}
We invoke Lemma~\ref{th:existence discrete minimizer} to obtain a sequence $(\bm y_h)_{h>0} \subset \mathbb A_h$ of discrete minimizers of $E_h$ over $\mathbb A_h$.  
Therefore, there exists $C > 0$, independent of $h > 0$ such that $E_h(\bm y_h) \le C$.
We invoke Proposition \ref{th:compactness} to get $\bm y \in \mathbb A$ such that $\bm y_h \to \bm y$ strongly in $L^2(\Omega)$ and $\nabla_h \bm y_h \to \nabla \bm y$ strongly in $L^2(\Omega)$.
We then apply Proposition \ref{th:liminf} to establish that
\[ E(\bm y) \le \liminf_{h \to 0} E_h(\bm y_h). \]
Next, we apply Proposition \ref{th:limsup} with a sequence $(\bm w_h)_{h > 0} \in \mathbb A_h$ such that
\[ \limsup_{h \to 0} E_h(\bm w_h) \le E(\bm y). \]
Then we use the fact that $\bm y_h$ is a global minimizer of $E_h$ to get that $E_h(\bm y_h) \le E_h(\bm w_h)$.
Then one has
\[ E(\bm y) \le \liminf_{h \to 0} E_h(\bm y_h) \le \limsup_{h \to 0} E_h(\bm w_h) \le E(\bm y). \] 
This implies that
\[ \lim_{h \to 0} E_h(\bm y_h) = E(\bm y). \qedhere \]
\end{proof}

\section{Numerical tests}
\label{sec:numerics}
We apply the method developed in this paper to two test cases from \cite{bartels2013approximation,bonito2021dg}.
In our experiments, we select the penalty parameters $\eta_1 = \eta_2 = 100$.
These parameters are much smaller than in \cite{bonito2021dg}, which improves the conditioning of the linear systems.
All the numerical tests have been implemented using \textit{Firedrake} \cite{Dalcin2011,FiredrakeUserManual}, which itself relies on \textit{PETSc} \cite{petsc-user-ref,petsc-efficient}.

\subsection{Practical algorithm}
In practice, numerical errors arising from floating-point arithmetic or inexact subproblem solves accumulate, causing violations to the constraint $\nabla_h \bm y_h^k(\bm x_T) \in St(3,2)$ for all $T \in \mathcal T_h$ implied by~\eqref{eq:stay in Stiefel}.
Fortunately, this numerical issue can easily be avoided by replacing $\nabla_h \bm y^k_h(\bm x_T)$ in~\eqref{eq:discrete problem} by an approximation $G_h^k \in M_h$ guaranteed to belong to $St(3,2)$ at each $T \in \mathcal T_h$.

Given such an approximate gradient $G_h^{k} \approx \nabla_h \bm y_h^{k}$, the practical algorithm involves first finding $(\bm y_h^{k+1}, \mu^{k+1}_h, \gamma_h^{k+1}) \in V_h \times M_h \times \Upsilon_h$ such that 
\begin{subequations}
\label{eq:modified}
\begin{alignat}{3}
  a_h(\bm{y}^{k+1}_h,\bm w_h) + (\mu^{k+1}_h,\nabla_h \bm w_h) + l_h(G_h^k;\nabla_h \bm w_h, \gamma_h^{k+1})
  &=
  (\bm f, \bm w_h) + F_h(\bm w_h),
  &&~~ \forall \bm w_h \in V_h, \\
  \sum_{T \in \mathcal T_h} (\Exp_{G_h^k(\bm x_T)} (\tau \mu_h^{k+1}),v_h)_{L^2(T)} - (\nabla_h \bm y_h^{k+1}, v_h)
  &= 0,
  &&~~ \forall v_h \in M_h,\\
  l_h(G_h^k;\mu_h^{k+1}, \zeta_h) &= 0, &&~~ \forall \zeta_h \in \Upsilon_h.
 \end{alignat}
\end{subequations}
After solving~\eqref{eq:modified}, the next approximate gradient is defined
\begin{equation}
\label{eq:G_update}
  G_h^{k+1}(\bm x_T)  := \Exp_{G_h^k(\bm x_T)}(\tau \Pi_{G_h^k(\bm x_T)}\mu_h^{k+1}(\bm x_T)), \quad \forall T \in \mathcal T_h,
\end{equation}
where $\Pi_{G_h^k(\bm x_T)} \colon \mathbb R^{2 \times 3} \to T_{G_h^k(\bm x_T)} St(2,3)$ is the orthogonal projection given in~\eqref{eq:projection}.
We summarize the process in Algorithm~\ref{algo}.

\begin{algorithm}
	\caption{Proximal Galerkin method for the isometry constraint}
 \label{algo}
	\begin{algorithmic}[1]
\Require Stopping tolerance $\texttt{tol.}>0$, initial guess $(\bm y_h^0, G_h^0) \in V_h \times M_h$.
        \State Set $k=-1$.
        \Repeat
        \State Assign $k \leftarrow k+1$.
        \State Compute a solution $(\bm y_h^{k+1}, \mu_h^{k+1}, \gamma_h^{k+1}) \in V_h \times M_h \times \Upsilon_h$ of~\eqref{eq:modified} via Newton's method.
        \State Define $G_h^{k+1}$ via~\eqref{eq:G_update}.
		\Until{$\tau^{-1} \left| E_h(\bm y_h^k) - E_h(\bm y_h^{k+1}) \right| < \texttt{tol.}$}
        \State Return $\bm y_h^k$. 
	\end{algorithmic} 
\end{algorithm}

Note that the convergence criterion in Algorithm~\ref{algo} is different from the criterion suggested in Remark~\ref{rk:convergence}.
This is only because we want to be able to compare our algorithm with the method from \cite{bonito2021dg}.
In practice, we still advocate for stopping based on the norm of $\mu_h^{k+1}$.

\subsection{Vertical load on a square domain}
In this test, the domain is $\Omega = (0,4)^2$ and $\Gamma^D = [\{0\} \times (0,4)] \cup [(0,4) \times \{0\}]$.
The Dirichlet boundary conditions are $\bm y_D(x_1,x_2) = (x_1, x_2, 0)$ and $\bm G_D = (n_1, n_2, 0)$, where $\bm n = (n_1,n_2)$ is the outward unit normal to $\partial \Omega$.
An upward pointing force $\bm f$ is applied everywhere on the domain.
Two different values of the magnitude of $\bm f$ are considered in the following test case.
We use as initial guess $\bm y^0(x_1,x_2) = (x_1,x_2,0) \in \mathbb A$.

We compare Algorithm \ref{algo} to the method from \cite{bonito2021dg} on a given set of structured meshes.
The only difference in the results reported for the method from \cite{bonito2021dg} is that we implement the modifications from \cite{MR4839135} to impose the isometry constraint at the barycenter of cells.
This leads to the following definition of the isometry defect:
\[ D_h(\bm y_h) := \max_{T \in \mathcal T_h} |\nabla_h \bm y_h^\tr \nabla_h \bm y_h-I|(\bm x_T),
\qquad \bm y_h \in V_h. \]
Note that the pseudo-time step for the gradient flow method in \cite{bonito2021dg} is $\tau = h$ in order to control the isometry defect.
This choice is widely advocated for in the literature for that class of algorithms and also followed here.
We stop both algorithms once the energetic stopping criterion $\tau^{-1} \left| E_h(\bm y_h^k) - E_h(\bm y_h^{k+1}) \right| < 10^{-4}$ is attained; cf.\ line 6 of Algorithm~\ref{algo}.
The subiteration counts reported below represent the cumulative number of Newton steps taken across the subproblem solves in line~4 of Algorithm~\ref{algo}.

\subsubsection{Weak volume force}
The upward pointing force $\bm f = 2.5 \cdot 10^{-2}\bm e_3$ is applied to the domain.
We fix $\tau = 2.0$ in Algorithm~\ref{algo} since we do not have a mesh-dependent restriction on the pseudo-time step size.
Table \ref{tab:grad flow 1} reports the results for the gradient flow method, whereas Table \ref{tab:PPG 1} reports the results from Algorithm \ref{algo}.

\begin{table}[htbp]
\centering
\begin{tabular}{|c|c|c|c|c|}
\hline
\# of Cells & dofs & $E_h(\bm y_h)$ & $D_h(\bm y_h)$ & Iterations  \\
\hline
400 & 8,400 & $-9.05 \cdot 10^{-3}$ & $7.58 \cdot 10^{-4}$ & 16 \\ \hline
1,600 & 33,600 & $-7.88 \cdot 10^{-3}$ & $3.30 \cdot 10^{-4}$ & 41 \\ \hline
6,400 & 134,400 & $-6.00 \cdot 10^{-3}$ & $1.26 \cdot 10^{-4}$ & 70 \\ \hline
14,400 & 302,400 & $-4.84 \cdot 10^{-3}$ & $7.02 \cdot 10^{-5}$ & 71 \\ \hline
\end{tabular}
\caption{Weak volume force.
Gradient flow from \cite{bonito2021dg} with $\tau = h$.}
\label{tab:grad flow 1}
\end{table}

\begin{table}[htbp]
\centering
\begin{tabular}{|c|c|c|c|c|c|}
\hline
\# of Cells & dofs & $E_h(\bm y_h)$ & $D_h(\bm y_h)$ & Iterations  & Subiterations \\
\hline
400 & 10,800 & $-9.80 \cdot 10^{-3}$ & $3.24 \cdot 10^{-14}$ & 4 & 14 \\ \hline
1,600 & 43,200 & $-9.49 \cdot 10^{-3}$ & $6.37 \cdot 10^{-14}$ & 4 & 15 \\ \hline
6,400 & 172,800 & $-8.59 \cdot 10^{-3}$ & $1.45 \cdot 10^{-13}$ & 3 & 12 \\ \hline
14,400 & 388,800 & $-7.82 \cdot 10^{-3}$ & $1.99 \cdot 10^{-13}$ & 3 & 13 \\ \hline
\end{tabular}
\caption{Weak volume force. Algorithm \ref{algo} with $\tau = 2$.}
\label{tab:PPG 1}
\end{table}

Observe that the standard gradient flow method requires an increasing number of iterations as the mesh is refined.
This aligns with the fact that the pseudo-time step $\tau$ decreases with the mesh size.
On the other hand, the total number of iterations and subiterations in Algorithm~\ref{algo} is uniformly bounded, due in part to the fact that $\tau$ is constant.
Next, notice that the isometry defect indeed remains on the order of machine error with Algorithm~\ref{algo}.
Although $D_h(\bm y_h)$ decreases with $h$ for the gradient flow method, the convergence rate is low, leading to a decrease of only one order of magnitude after four mesh refinements.
Finally, note that the final values of $E_h(\bm y_h)$ are significantly lower at the final iteration with Algorithm~\ref{algo} than with the standard gradient flow method.
This shows that enforcing the isometry constraint at each iteration does not cause Algorithm~\ref{algo} to be less efficient in minimizing the bending energy.

\subsubsection{Strong volume force}
The upward pointing force $\bm f = \bm e_3$ is applied to the domain.
In this case, we set $\tau = 5.0 \cdot 10^{-2}$ in Algorithm~\ref{algo} but continue to use $\tau = h$ for the standard gradient flow method.
Table~\ref{tab:grad flow 2} reports the results for the gradient flow method, whereas Table~\ref{tab:PPG 2} reports the results from Algorithm~\ref{algo}.

Although we did not change the step size for the gradient flow method, we terminated those computations after 1000 iterations due to slow convergence to the stopping criterion.
Instead, we report the weighted energy increment at the 1000th iteration.  
In contrast, Algorithm~\ref{algo} converges significantly faster, and $D_h(\bm y_h)$ remains at machine precision, as expected.  

\begin{table}[htbp]
\centering
\begin{tabular}{|c|c|c|c|c|c|}
\hline
\# of Cells & dofs & $E_h(\bm y_h)$ & $D_h(\bm y_h)$ & Iterations & $\tau^{-1}(E_h(\bm y_h^k) - E_h(\bm y_h^{k+1}))$  \\
\hline
400 & 8,400 & $-12.3$ & $7.56 \cdot 10^{-1}$ & 1,000 & $6.01 \cdot 10^{-3}$  \\ \hline
1,600 & 33,600 & $-7.06$ & $2.69 \cdot 10^{-1}$ & 1,000 & $5.13 \cdot 10^{-3}$ \\ \hline
6,400 & 134,400 & $-4.11$ & $8.45 \cdot 10^{-2}$ & 1,000 & $4.05 \cdot 10^{-3}$ \\ \hline
14,400 & 302,400 & $-3.46$ & $4.57 \cdot 10^{-2}$ & 1,000 & $5.15 \cdot 10^{-3}$ \\ \hline
\end{tabular}
\caption{Strong volume force. Gradient flow from \cite{bonito2021dg} with $\tau = h$.}
\label{tab:grad flow 2}
\end{table}

\begin{table}[htbp]
\centering
\begin{tabular}{|c|c|c|c|c|c|}
\hline
\# of Cells & dofs & $E_h(\bm y_h)$ & $D_h(\bm y_h)$ & Iterations  & Subiterations \\
\hline
400 & 10,800 & $-5.41$ & $2.48 \cdot 10^{-14}$ & 112 & 236 \\ \hline
1,600 & 43,200 & $-4.21$ & $7.35 \cdot 10^{-14}$ & 116 & 306 \\ \hline
6,400 & 172,800 & $-3.50$ & $1.25 \cdot 10^{-13}$ & 99 & 281 \\ \hline
14,400 & 388,800 & $-3.22$ & $2.44 \cdot 10^{-13}$ & 93 & 264 \\ \hline
\end{tabular}
\caption{Strong volume force. Algorithm \ref{algo} with $\tau = 5.0 \cdot 10^{-2}$.}
\label{tab:PPG 2}
\end{table}

\subsection{Buckling of a strip}
Our final experiment follows \cite{bartels2013approximation,bartels2013finite,bonito2021dg}.  
To this end, we set $\Omega = (-2,2) \times (0,1)$ and $\Gamma^D = \{-2,2\} \times (0,1)$.  
The Dirichlet boundary conditions are $\bm y_D(x_1,x_2) = (x_1 - \mathrm{sign}(x_1)\,1.4, x_2, 0)$ and $\bm G_D = (n_1,n_2,0)$, where $\bm n = (n_1,n_2)$ is the outward unit normal to $\partial \Omega$.  
To address non-uniqueness, we add a small upward force $\bm f = 10^{-5}\bm e_3$.  
We use as initial guess $\bm y^0(x_1,x_2) = (x_1,x_2,0) \in \mathbb A$.
The boundary condition $\bm y_D$ is imposed quasi-statically using pseudo-time steps of size $\delta t = 10^{-3}$, which is significantly larger than the $\delta t = 5.0 \cdot 10^{-5}$ used in \cite{bonito2021dg}, while $\bm G_D$ is enforced at each step.  
The mesh consists of 3,740 elements and 100,980 degrees of freedom.  
We set $\texttt{tol.} = 10^{-3}$ in Algorithm~\ref{algo}.  
As in \cite{bonito2021dg}, we display snapshots of the deformation in Figure~\ref{fig:buckling}.  
In our opinion Figures~\ref{fig:sub1}--\ref{fig:sub3} appear to capture the buckling process more accurately than those in \cite{bonito2021dg}.  
Moreover, the final isometry defect is negligible, with $D_h(\bm y_h) = 1.92 \cdot 10^{-13}$, reflecting the design of Algorithm~\ref{algo} to enforce the constraint~\eqref{eq:IsometryConstraint} exactly at each element midpoint.

\begin{figure}[htbp]
  \centering
  \colorbox{black!06}{
  % Row 1
  \begin{subfigure}[b]{0.46\textwidth}
    \centering
    \colorbox{black!05}{\includegraphics[width=\linewidth]{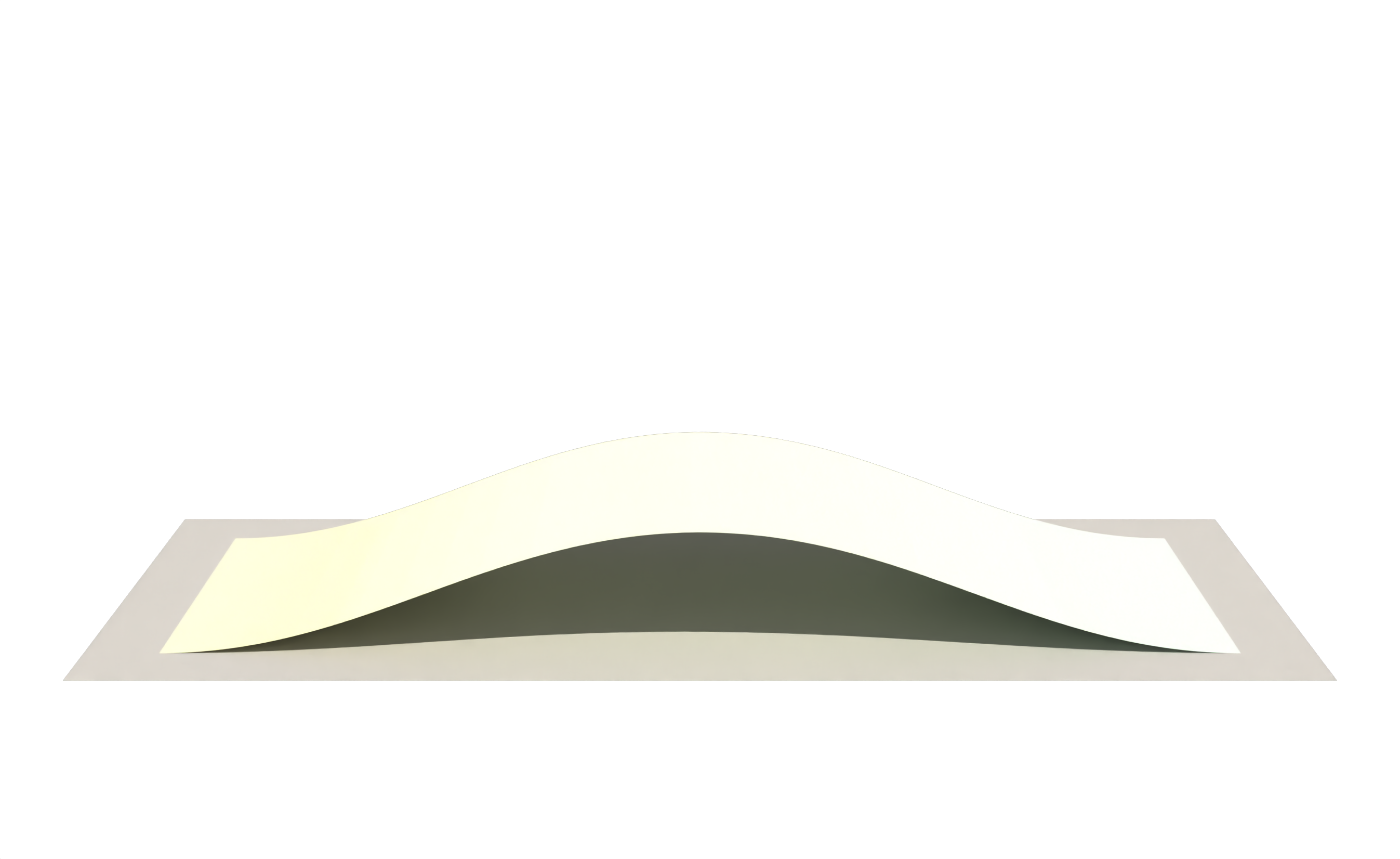}}
    \caption{$t=0.05$}
    \label{fig:sub1}
  \end{subfigure}\hfill
  \begin{subfigure}[b]{0.46\textwidth}
    \centering
    \colorbox{black!05}{\includegraphics[width=\linewidth]{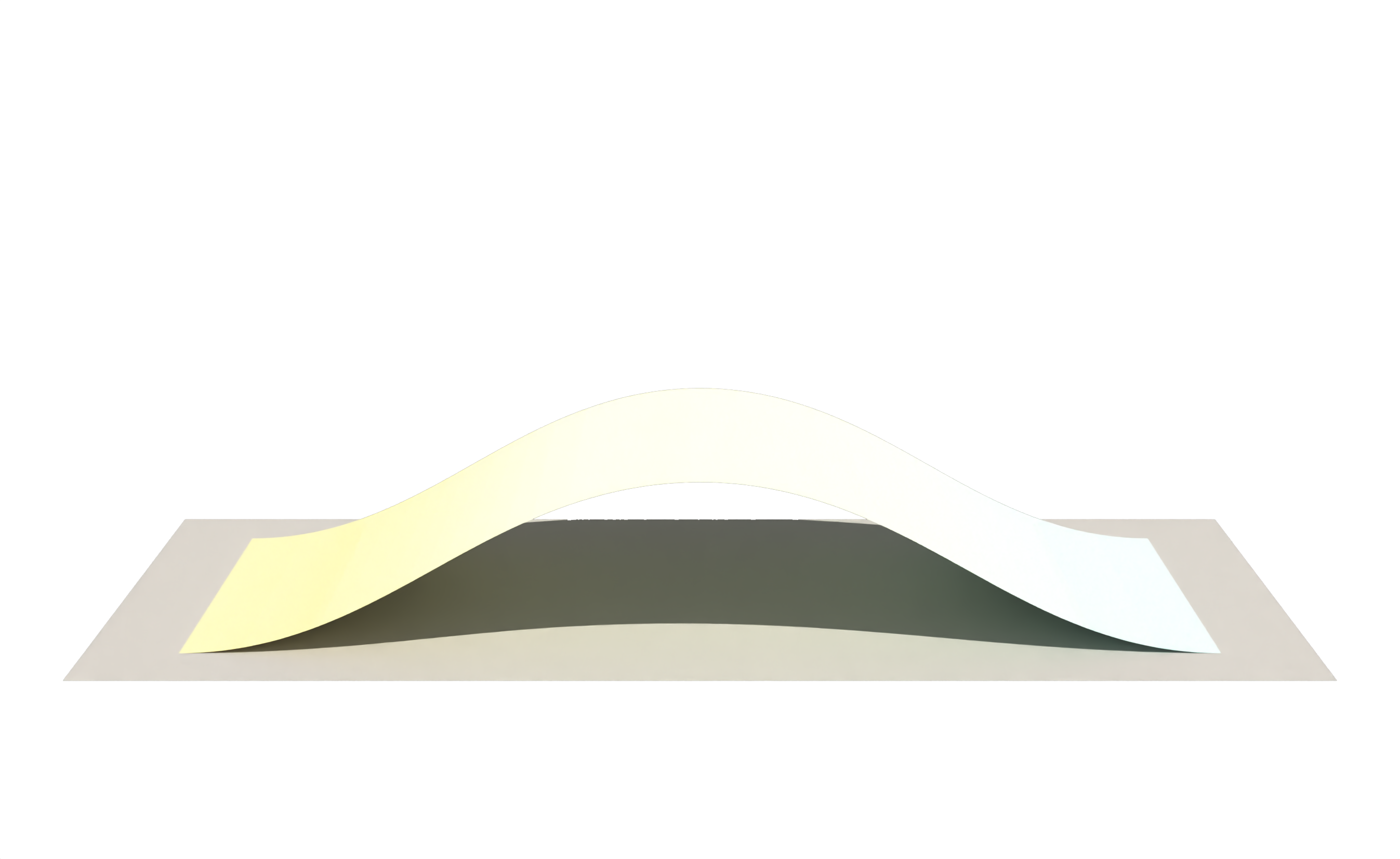}}
    \caption{$t=0.1$}
    \label{fig:sub2}
  \end{subfigure}
  \begin{subfigure}[b]{0.06\textwidth}
    \centering
    \colorbox{black!05}{\phantom{\includegraphics[width=0.95\linewidth]{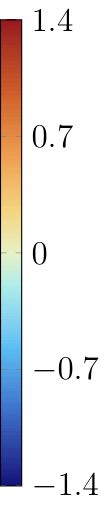}}}
    \caption*{~}
  \end{subfigure}
}
  \vspace*{-1.1em}

  \colorbox{black!05}{
  % Row 2
  \begin{subfigure}[b]{0.46\textwidth}
    \centering
    \colorbox{black!05}{\includegraphics[width=\linewidth]{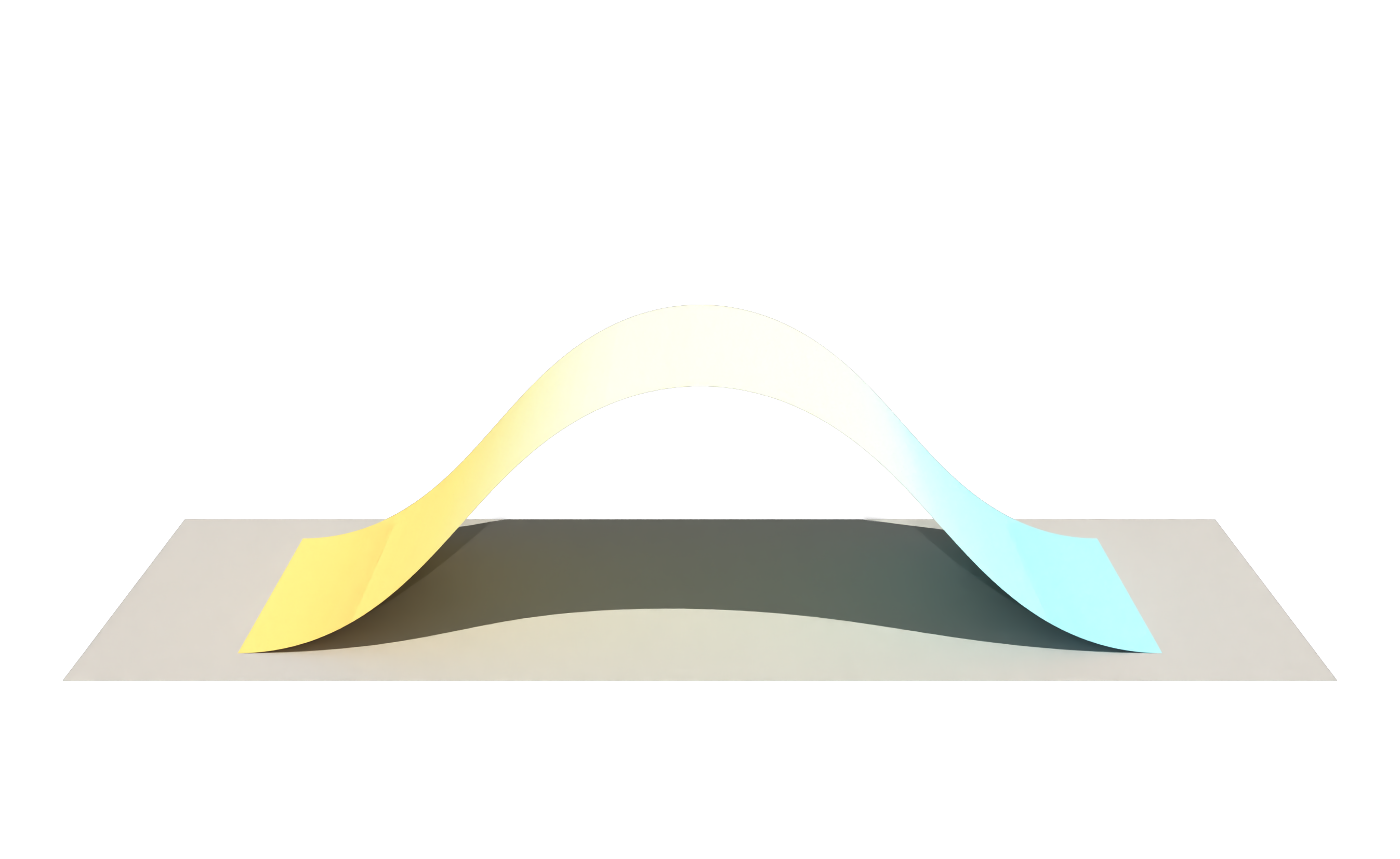}}
    \caption{$t=0.25$}
    \label{fig:sub3}
  \end{subfigure}\hfill
  \begin{subfigure}[b]{0.46\textwidth}
    \centering
    \colorbox{black!05}{\includegraphics[width=\linewidth]{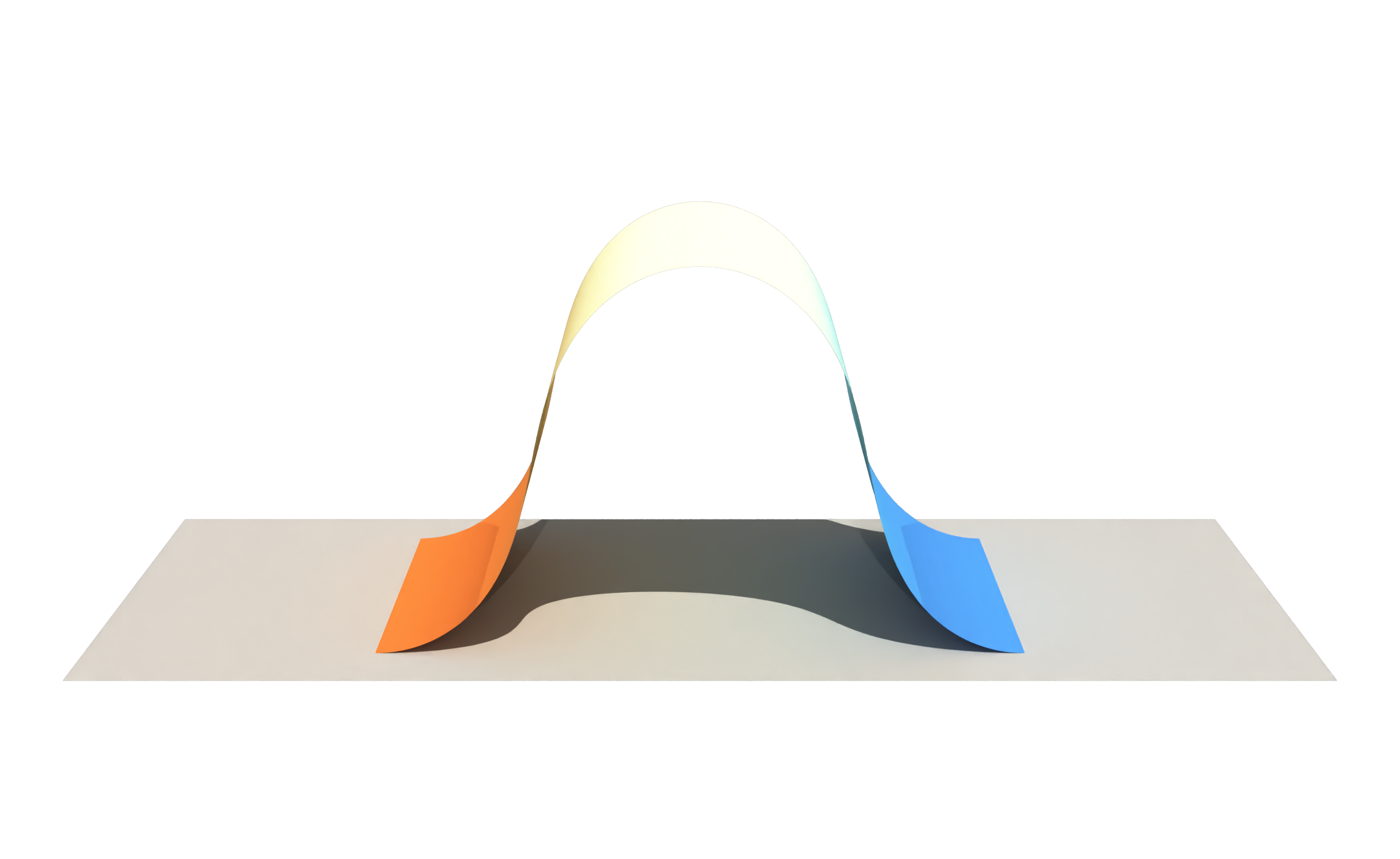}}
    \caption{$t=0.6$}
    \label{fig:sub4}
  \end{subfigure}
  \begin{subfigure}[b]{0.06\textwidth}
    \centering
    \colorbox{black!05}{\includegraphics[width=0.95\linewidth]{colorbar/colorbar.png}}
    \caption*{~}
    \label{fig:colorbar}
  \end{subfigure}
}

  \vspace*{-0.1em}

  % Row 3
  \colorbox{black!05}{
  \begin{subfigure}[b]{0.46\textwidth}
    \centering
    \colorbox{black!05}{\includegraphics[width=\linewidth]{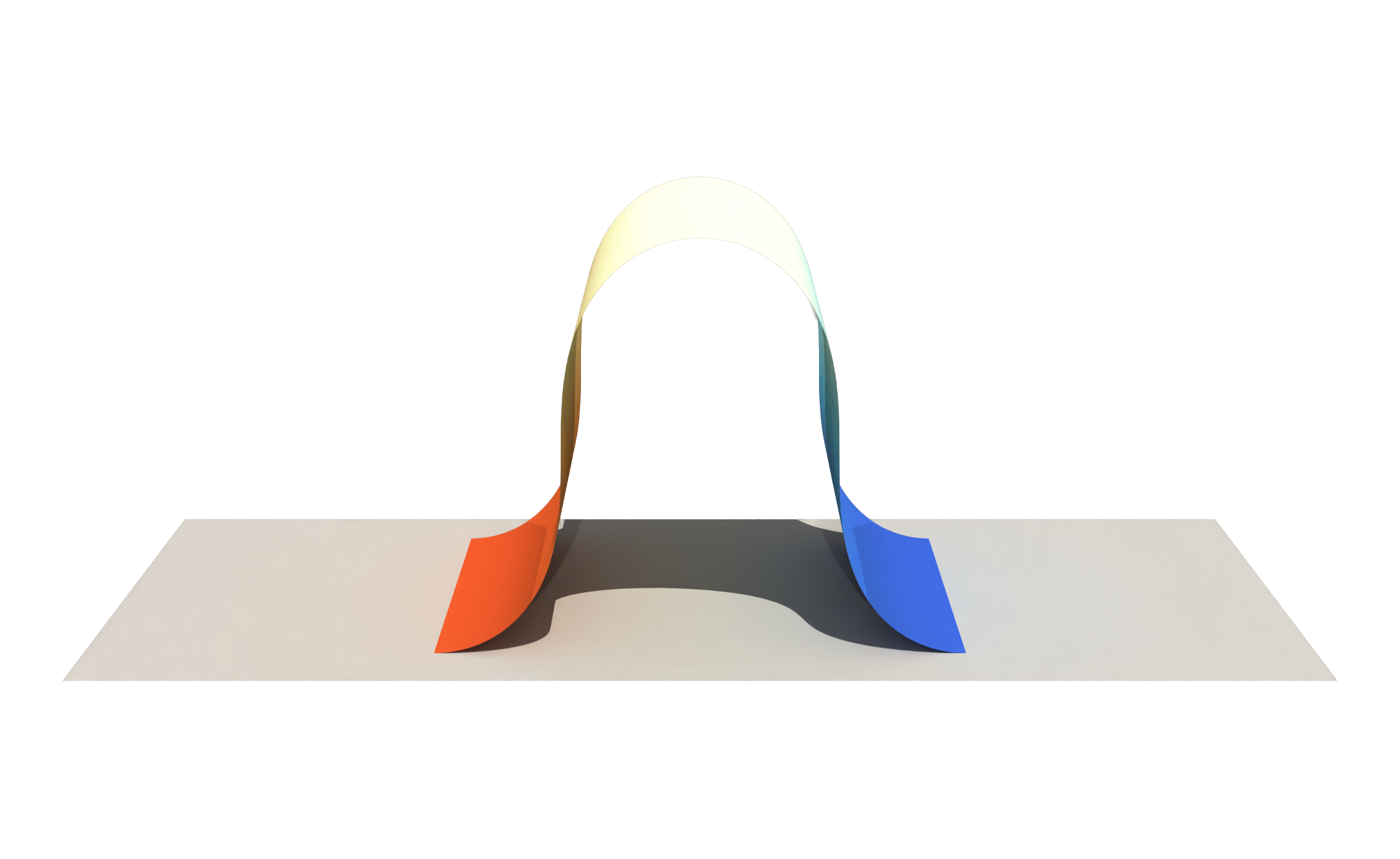}}
    \caption{$t=0.75$}
    \label{fig:sub5}
  \end{subfigure}\hfill
  \begin{subfigure}[b]{0.46\textwidth}
    \centering
    \colorbox{black!05}{\includegraphics[width=\linewidth]{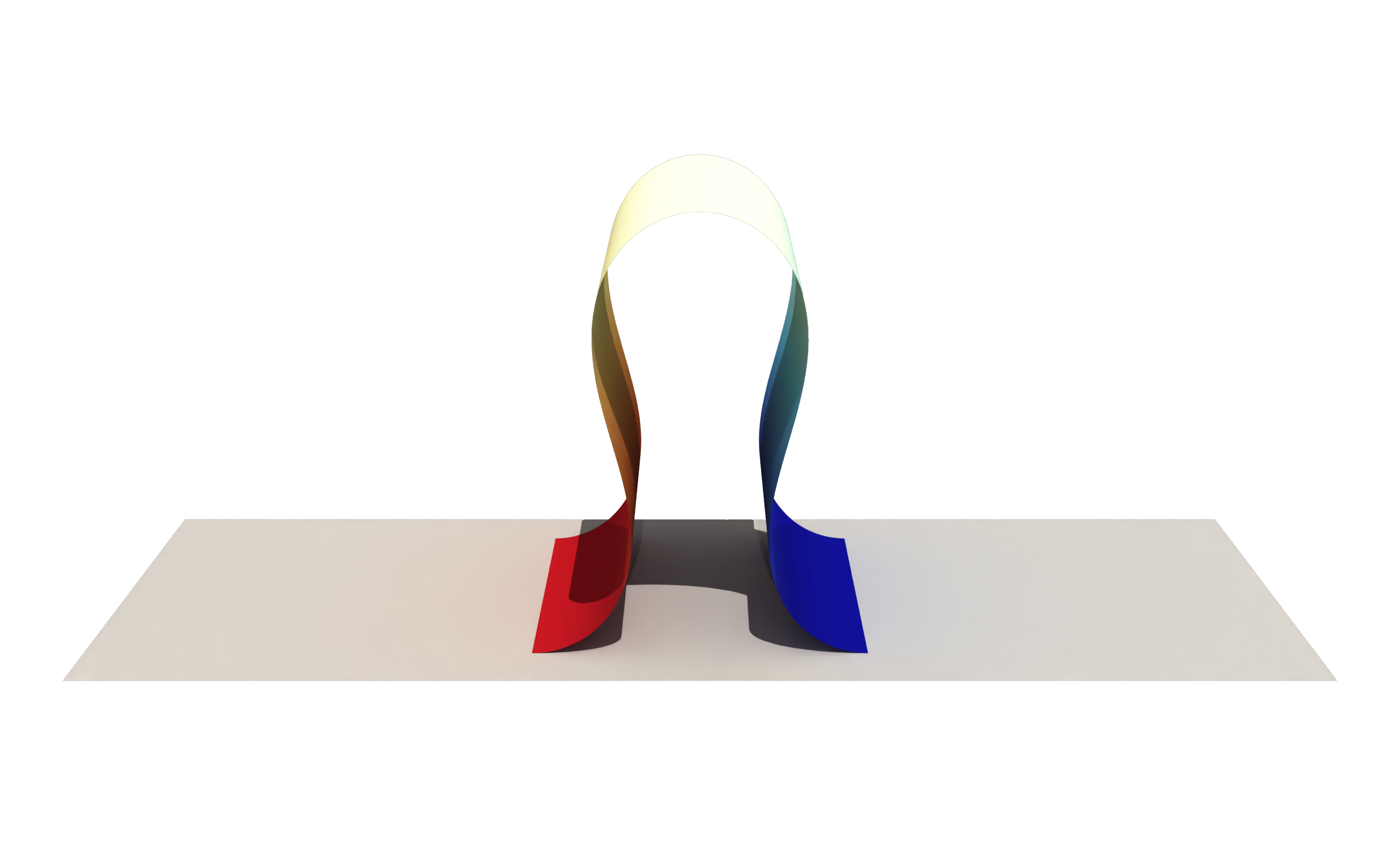}}
    \caption{$t=1$}
    \label{fig:sub6}
  \end{subfigure}
  \begin{subfigure}[b]{0.06\textwidth}
    \centering
    \colorbox{black!05}{\phantom{\includegraphics[width=0.95\linewidth]{colorbar/colorbar.png}}}
    \caption*{~}
  \end{subfigure}
}
  \caption{Buckling test based on \cite{bonito2021dg} with color depicting the amount of horizontal displacement.}
  \label{fig:buckling}
\end{figure}

\section*{Code availability}
The code is freely available at \url{https://github.com/marazzaf/iso_constraint_PG}.

\section*{Acknowledgements}
We thank Patrick E.\ Farrell and J{\o}rgen S.\ Dokken for help developing and debugging the exponential matrix functions used in our computations.
We also thank Dohyun Kim for his comments on the manuscript.

\section*{Funding}
BK was supported in part by the Center for Information Geometric Mechanics and Optimization (CIGMO), a PSAAP-IV Focused Investigatory Center funded by the U.S.\ Department of Energy, National Nuclear Security Administration under Award Number DE-NA0004261.
BK was also supported in part by the Alfred P.\ Sloan Foundation via a Sloan Research Fellowship in Mathematics.

\noindent FM is partially supported by the National Science Foundation under grant DMS-2409926.

\bibliographystyle{plain}
\bibliography{bib}

\appendix

\section[{APPENDIX}]{Proof of Lemma \protect\ref{th:grad distance}}
\label{sec:proof}
Before giving a proof of Lemma \ref{th:grad distance}, we recall a few definitions and results from \cite[Chapter~6]{MR3887684} and \cite[Chapter~9]{do1992riemannian}.

\begin{definition}[Energy of a curve]
\label{def:energy curve}
Let $(\mathcal M, g)$ be a Riemannian manifold.
Let $a < b$ and $\gamma : [a,b] \to \mathcal M$ be a piecewise-smooth curve.
The energy of $\gamma$ is defined to be
\[ E( \gamma) := \frac12 \int_a^b |\dot{\gamma}|_{\gamma(t)}^2 \dd t. \]
\end{definition}

\begin{definition}[Variation of a curve]
\label{def:variation}
Let $\gamma : [a,b] \to \mathcal M$ be a smooth curve.
A variation of $\gamma$ is a smooth family of curves $\Gamma:J \times [a,b] \to \mathcal M$ such that $J \ni 0$ is an open interval and $\Gamma(0,t) = \gamma(t)$ for each $t \in [a,b]$.

\end{definition}

\begin{definition}[Variation field]
\label{def:variation field}
Let $\Gamma$ be a variation of a smooth curve $\gamma$.
The variation field of $\Gamma$ is the smooth vector field $V: [a,b] \to T \mathcal M$ such that $V(t) = \partial_s \Gamma(0,t)$ along $\gamma$.
\end{definition}

\noindent We now state a well-known first variation formula; cf.~\cite[p.~195]{do1992riemannian}.
Note that $\mathcal E(\gamma) = \frac12 \operatorname{Dist}(\gamma(a),\gamma(b))^2$ if $\gamma$ is a geodesic.

\begin{theorem}[First variation formula for the energy of a curve]
\label{th:first variation}
Let $\gamma:[a,b] \to \mathcal M$ be a smooth curve.
Let $\Gamma$ be a smooth variation of $\gamma$ with $V(t) := \partial_s \Gamma(0,t)$.
The G\^{a}teaux derivative of $\mathcal E$ can be written
\[
    \mathcal E^\prime(\gamma)(V)
    =
    -\int_a^b g(V(t), D_t \dot{\gamma}) \dd t - g(V(a), \dot{\gamma}(a)) + g(V(b), \dot{\gamma}(b)).
\]
\end{theorem}

\noindent In the expression above, $D_t \dot{\gamma}$ denotes the covariant derivative of $\dot{\gamma}$ along $\gamma$, which vanishes along geodesics; see, e.g., \cite[p.~103]{MR3887684}.
Thus, if $\gamma$ is a geodesic, then
\[
    \mathcal E^\prime(\gamma)(V)
    =
    g(V(b), \dot{\gamma}(b)) - g(V(a), \dot{\gamma}(a))
\]
depends only on the variation field at the endpoints of the curve.
This property is leveraged in the proof of Lemma~\ref{th:grad distance}, below.

\begin{proof}[Proof of Lemma~\ref{th:grad distance}]
Let $p \in \mathcal M$.
For all $q \in \mathcal M$, we define $f(q) = \frac12 \operatorname{Dist}(q,p)^2$.
If $q$ is sufficiently close to $p$, we let $\gamma_q \colon [0,1] \to \mathcal M$ be the unique geodesic such that $\gamma_q(0) = q$ and $\gamma_q(1) = p$.
Note that $\dot{\gamma}_q(0) = \Log_q(p)$ by Definition~\ref{def:log} and $\mathcal E(\gamma_q) = f(q)$ because $\gamma_q$ is a geodesic.

We now consider a variation of $\gamma_q$, denoted by $\Gamma_q:J \times [0,1] \to \mathcal M$, with $\Gamma_q(s,1) = p$ for all $s \in J$.
Following the notation of Definition \ref{def:variation field}, we write the corresponding variation field as $V_q$ and note that $V_q(1) = 0$.
We can now apply Theorem \ref{th:first variation}, along with the properties $D_t \dot{\gamma}_q = 0$, $V_q(1) = 0$, and $\dot{\gamma}_q(0) = \Log_q(p)$, to arrive at the identity
\begin{equation*}
    d f_q (V_q(0))
    =
    \mathcal E^\prime(\gamma_q)(V_q)
    =
    -g( V_q(0), \Log_q(p))
    .
\end{equation*}
Since $q$ and $V_q(0) \in T_q \mathcal M$ were arbitrary, one has
\[ g(\Grad f, \mu) = d f (\mu) = -g(\Log_{(\cdot)}(p), \mu), \]
for all $\mu \in T \mathcal M$, as necessary.
\end{proof}

\section[{APPENDIX}]{Proof of Lemma \protect\ref{th:Lipschitz}}
\label{sec:proof 2}

\begin{proof}[Proof of Lemma \ref{th:Lipschitz}]
Let $U \in St(3,2)$.
In the following $A \lesssim B$ is used to denote $A \le c B$, where $c > 0$ is a constant independent of $A$, $B$ and $U \in St(3,2)$.

Let $W, G \in T_U St(3,2)$.
We define $A := WU^\tr - U W^\tr$, $B := GU^\tr - U G^\tr$, $E := -U^\tr W$ and $F := -U^\tr G$.
Recalling \eqref{eq:exponential map}, one has
\[ D \Exp_U(W)(G) =  \left. \frac{d}{dt} \exp(A+tB) \right|_{t=0} U \exp(E) + \exp(A) U \left. \frac{d}{dt} \exp(E+tF) \right|_{t=0}. \]
Let $\tilde W \in T_U St(3,2)$.
We define $\tilde A := \tilde WU^\tr - U \tilde W^\tr$ and $\tilde E := - U^\tr \tilde W$.
One has
\[ \begin{aligned}
[D \Exp_U(W) - D \Exp_U(\tilde W)](G) &= \left( \left. \frac{d}{dt} \exp(A+tB) \right|_{t=0} - \left. \frac{d}{dt} \exp(\tilde A+tB) \right|_{t=0} \right) U \exp(E)  \\
& \quad + \left. \frac{d}{dt} \exp(\tilde A+tB) \right|_{t=0} U [\exp(E) - \exp(\tilde E)] \\
& \quad + \exp(A) U \left( \left. \frac{d}{dt} \exp(E+tF) \right|_{t=0} - \left. \frac{d}{dt} \exp(\tilde E+tF) \right|_{t=0} \right) \\
& \quad + [\exp(A) - \exp(\tilde A)] U \left. \frac{d}{dt} \exp(\tilde E+tF) \right|_{t=0}.
\end{aligned} \]
Note that $A$ and $\tilde A$ are skew-symmetric.
According to \cite[Exercise~8-29,Proposition~20.8]{MR2954043} the matrix exponential applied to $A$ and $\tilde A$ produces orthogonal matrices, which are bounded.
A similar result holds for $E$ and $\tilde E$.
Applying \cite[Theorem 3(b)]{haber2018notes}, one has
\[ \left. \frac{d}{dt} \exp(\tilde A+tB) \right|_{t=0} = \int_0^1 \exp((1-s) \tilde A) B \exp(s \tilde A) \dd s.  \]
Thus,
\[ \left| \left. \frac{d}{dt} \exp(\tilde A+tB) \right|_{t=0}  \right|\le |B| \int_0^1 |\exp((1-s) \tilde A) | |\exp(s \tilde A)| \dd s \lesssim |G|.  \]
A similar result holds for $\tilde E + tF$.
Therefore,
\[ \begin{aligned}
\left| \left. \frac{d}{dt} \exp(A+tB) \right|_{t=0} - \left. \frac{d}{dt} \exp(\tilde A+tB) \right|_{t=0} \right| & \le |B| \int_0^1 |\exp((1-s)A) - \exp((1-s) \tilde A)| \dd s \\
& \quad + |B| \int_0^1 |\exp(s A) - \exp(s \tilde A)| \dd s, \\
& \lesssim |G| \int_0^1 (1-s) |A - \tilde A| + s |A - \tilde A| \dd s \lesssim |G|  |A - \tilde A|,
\end{aligned} \]
since $\exp$ is smooth and thus Lipschitz.
With a similar reasoning as above, one has
\[ \left| \left. \frac{d}{dt} \exp(E+tF) \right|_{t=0} - \left. \frac{d}{dt} \exp(\tilde E+tF) \right|_{t=0} \right| \lesssim  |G| |E - \tilde E|. \]
Summarizing the results above, one can conclude that 
\[ \| D \Exp_U(W) - D \Exp_U(\tilde W) \| \lesssim |W - \tilde W|. \]

\end{proof}

\end{document}